\documentclass{amsart}
\usepackage[utf8]{inputenc}
\usepackage{amsfonts}
\usepackage{amsmath}
\numberwithin{equation}{section}
\usepackage{amssymb}
\usepackage{amsthm}
\usepackage{mathrsfs}
\usepackage{hyperref}
\usepackage[capitalise]{cleveref}
\usepackage{tikz}
\usetikzlibrary{arrows.meta,positioning,calc,fit,backgrounds}
\usepackage[top=3cm,bottom=3cm,left=2.5cm,right=2.5cm,headsep=10pt,letterpaper]{geometry} % Page margins

\usepackage[
    backend=biber,
    style=alphabetic,
    sorting=nyt
]{biblatex}
\theoremstyle{plain}
\newtheorem{theorem}{Theorem}[section]
\newtheorem{proposition}[theorem]{Proposition}
\newtheorem{lemma}[theorem]{Lemma}
\newtheorem{corollary}[theorem]{Corollary}

\theoremstyle{definition}

\theoremstyle{remark}
\newtheorem{remark}[theorem]{Remark}

\allowdisplaybreaks[2]
\newcommand{\CP}{\mathbb{CP}}
\newcommand{\Hess}{\operatorname{Hess}}
\newcommand{\II}{\mathrm{I\!I}}
\newcommand{\Id}{\operatorname{Id}}
\newcommand{\cost}{\mathcal C}
\DeclareMathOperator{\Ric}{Ric}
\DeclareMathOperator{\tr}{tr}

\title{Log-Concavity and Fundamental Gaps on Complex Projective Domains}

\author[G. Cho]{Gunhee Cho}
\address{Department of Mathematics \\ Texas State University \\ San Marcos, TX 78666, USA}
\email{wvx17@txstate.edu}



\hypersetup{pdftitle={Log-Concavity and Fundamental Gaps on Complex Projective Domains},pdfauthor={Gunhee Cho}}

\begin{document}

\begin{abstract}
We prove quantitative log-concavity and fundamental-gap estimates for
Dirichlet Schr\"odinger operators on complex projective domains.
With holomorphic sectional curvature four, the boundary condition
$\II\ge\sqrt{44/5}\,g$ suffices in every complex dimension under
$\Hess V\ge2|\nabla V|g$ and $\operatorname{osc}V\le4n$.
It yields $\Hess(-\log\phi_1)>g$ and
$\lambda_2-\lambda_1>\max\{n+2+\sqrt{n^2+5n+11},\,2n+14/3\}$.
The main difficulty is to control the distinguished Hopf direction together
with a transverse Hessian block whose dimension grows with $n$.
We extend modified-Hessian first-contact arguments by retaining the
Hopf term and using longitudinal and transverse rotational tests.
A trace minimization reduces all gradient orientations to three extreme
directions. A dimension-scaled logarithmic barrier then compares their
contact costs with a single limiting profile. A radial spectral estimate
supplies the required quadratic-in-dimension spectral threshold from
a dimension-independent boundary curvature bound. For
$n\ge3$ these are, to our knowledge, the first log-concavity estimates
for domains in $\CP^n$ without symmetry assumptions. We also prove a sharp diameter bound and
Hessian levels linear in $n$ under larger dimension-free curvature
bounds.
\end{abstract}

\maketitle

\setcounter{tocdepth}{1}
\tableofcontents

\section{Introduction}
Log-concavity of a Dirichlet ground state links the geometry of its
domain to quantitative spectral estimates. On complex projective space,
the curvature distinguishes the complex line generated by each tangent
vector from its orthogonal complement. This distinction obstructs a
direct transfer of the Euclidean and spherical Hessian arguments.
The conjecture attributed to Zhiqin Lu in \cite{DSW19} asks whether the
first Dirichlet eigenfunction of every convex domain in $\CP^n$ is
log-concave.

We prove a quantitative form of log-concavity in every complex dimension
under an explicit boundary curvature bound independent of dimension,
without imposing symmetry on the domain. The result applies to
Schr\"odinger operators with a specified class of potentials and gives
explicit fundamental-gap estimates. The main point is a contact
inequality that controls both the Hopf direction and the growing
transverse Hessian block. Its reduction to three extreme gradient
directions and a single normalized barrier allows all dimensions to
be treated together. Our curvature hypothesis is stronger than
geodesic convexity; the general conjecture remains open.

Let $M=\CP^n$ with the Fubini--Study metric of holomorphic sectional
curvature four. Throughout, $\Omega\subset M$ is a smooth connected
proper geodesically convex domain with nonempty embedded boundary.
With inward unit normal $\nu$, our convention is
$\II(X,Y)=-\langle\nabla_X\nu,Y\rangle$.
The condition $\II\ge\kappa g$ bounds every principal curvature from
below by $\kappa$; it does not refer to a sum of principal curvatures.
Let $V\in C^\infty(\overline\Omega;\mathbb R)$ and consider
\begin{equation}\label{eq:eigen}
 (-\Delta+V)\phi_1=\lambda_1\phi_1,\qquad
 \phi_1>0\text{ in }\Omega,\quad \phi_1|_{\partial\Omega}=0,
 \qquad f=-\log\phi_1,\quad s=|\nabla f|^2.
\end{equation}
Here $\Delta=\operatorname{div}\nabla$ and the Dirichlet realization
of $-\Delta+V$ has domain $H^2(\Omega)\cap H_0^1(\Omega)$ in
$L^2(\Omega,d\mathrm{vol}_g)$.
Write $\lambda_1<\lambda_2\le\cdots$ for its spectrum,
$\gamma=\lambda_2-\lambda_1$, and $\lambda_1^0$ for the first
Dirichlet eigenvalue of $-\Delta$. Eigenfunctions are real and
$L^2$-normalized. Symmetric tensors are identified with self-adjoint
endomorphisms using $g$, and $A>B$ means pointwise positive definiteness
of $A-B$. We use the shift-invariant spectral quantity
\begin{equation}\label{eq:endpointintro}
 \ell(V)=\lambda_1(-\Delta+V)-\max_{\overline\Omega}V.
\end{equation}
The variational bound $\ell(V)\ge\lambda_1^0-\operatorname{osc}_\Omega V$
gives an elementary sufficient condition for every lower bound on
$\ell(V)$ used below.

\begin{theorem}[A uniform curvature condition]\label{thm:uniformintro}
For every $n\ge1$, suppose
\[
 \II\ge\kappa_*g,\qquad \kappa_*:=\sqrt{44/5},\qquad
 \Hess V\ge2|\nabla V|g,\qquad \operatorname{osc}_\Omega V\le4n.
\]
Then $\Hess(-\log\phi_1)>g$ and
\[
 \lambda_2-\lambda_1>
 \max\left\{n+2+\sqrt{n^2+5n+11},\ 2n+\frac{14}{3}\right\}.
\]
For $n=1$ one has the stronger bound $\Hess(-\log\phi_1)>4g$.
\end{theorem}
The uniform level is $\kappa_*\simeq2.966$ (for orientation only), and is asymptotically optimal, as
$n\to\infty$, for converting a containing-ball estimate into the
spectral threshold of this particular barrier
(Remark~\ref{rem:leadinglimit}). This does not assert optimality of
the geometric hypothesis or of the barrier among other choices. For $n\ge2$, put $m=2n-1$ and
$L_m=11m^2/5+12m$; the dimension-dependent level
\begin{equation}\label{eq:intro-newkappa}
 \kappa_n=\max\left\{2,
 \frac{\sqrt{L_m+4n}+\sqrt{L_m-4}}{2(n+1)}\right\}<\kappa_*
\end{equation}
already suffices (Corollary~\ref{cor:improved-finitekappa}). Under the
hypotheses of Theorem~\ref{thm:uniformintro}, the diameter bound below
improves the gap to
$\gamma>n+2+y_*^{-1}+C_ny_*+E_ny_*^2$, where
$C_n=(n^2+5n+11)/4$, $E_n=(n^2+10n+45)/24$, $P_*=\pi^2/\arctan^2(2/\kappa_*)$,
$y_*=\min\{y_n,P_*^{-1}\}$, and $y_n>0$ solves $2E_ny_n^3+C_ny_n^2=1$
(Corollary~\ref{cor:freshuniformgap}).

\begin{theorem}[Larger quantitative Hessian levels]\label{thm:schall}
Let $n\ge1$, $\Hess V\ge2|\nabla V|g$, and
$\operatorname{osc}_\Omega V\le4n$. For $A\in\{64,256,6400\}$ put
\[
 m=2n-1,\qquad
 \kappa_{n,A}^{\mathrm{suff}}
 =\min\left\{\frac{\sqrt A\,m}{2\sqrt n},
             \frac{\sqrt{A m^2+4n}}{n}\right\}.
\]
The condition $\II\ge\kappa_{n,A}^{\mathrm{suff}}g$ gives,
respectively,
\begin{align*}
 A=64:&\quad \Hess f>\frac{4m}{7}g,
                  &&\gamma>\frac{30n+6}{7},\\
 A=256:&\quad \Hess f>5m g,
                  &&\gamma>22n-8,\\
 A=6400:&\quad \Hess f>50m g,
                  &&\gamma>202n-98.
\end{align*}
In particular, the dimension-independent conditions $\II\ge16g$,
$\II\ge32g$, and $\II\ge160g$, respectively, suffice.
Alternatively, $\II\ge2g$ and $\ell(V)\ge A m^2$, with the
same differential condition on $V$, imply the corresponding conclusions
without an oscillation hypothesis.
\end{theorem}

For $n\ge2$, the spectral slack in Theorem~\ref{thm:uniformintro}
permits nonconvex potentials, with $\Hess V$ bounded below by a negative multiple of the
slack (Corollary~\ref{cor:uniformnonconvex}).

The hypothesis on $V$ is satisfied by a natural class of potentials.
The class $\{V:\Hess V\ge2|\nabla V|g\}$ is a convex cone containing
the constants, since $|\nabla(V_1+V_2)|\le|\nabla V_1|+|\nabla V_2|$.
It also contains $c\,w_o$ for $c\ge0$ and $w_o=-\log\cos d(o,\cdot)$
whenever $\Omega\subset B_{\pi/8}(o)$. Indeed, with $\rho=d(o,\cdot)$
and $t=\tan\rho$, the eigenvalues of $\Hess\rho$ are $0$, $2\cot2\rho$
on $J\nabla\rho$ and $\cot\rho$ otherwise, so $\Hess w_o$ has
eigenvalues $1+t^2$, $1-t^2$ and $1$, while $|\nabla w_o|=t$; and
$1-t^2\ge2t$ for $t\le\sqrt2-1=\tan(\pi/8)$. Under
$\II\ge\kappa_*g$ the domain lies in a ball of radius
$\arctan(1/\kappa_*)<\pi/8$ by Drach's rolling theorem
\cite[Theorem~A]{Drach}, as in Lemma~\ref{lem:smooth-radial-bound},
because $1/\kappa_*^2=5/44<3-2\sqrt2$. For this $V$ one has
$\operatorname{osc}_\Omega V\le c\log\sqrt{1+\kappa_*^{-2}}$.

\begin{figure}[t]
\centering
\begin{tikzpicture}[
  box/.style={draw,rounded corners=2pt,align=center,font=\small,
              inner sep=4pt,minimum height=9mm},
  geo/.style={box,fill=blue!6},
  spec/.style={box,fill=orange!8},
  hess/.style={box,fill=green!7},
  arr/.style={-{Stealth[length=2mm]},thick},
  lab/.style={font=\scriptsize,align=center,midway}
]
\node[geo]  (A) {$\II\ge\kappa g$};
\node[geo,right=9mm of A] (B) {$\Omega\subset B_{\operatorname{arccot}\kappa}(o)$\\[1pt]
      $\operatorname{diam}\Omega\le\arctan(2/\kappa)$};
\node[spec,right=9mm of B] (C) {$\lambda_1^0>\bigl((n+1)\kappa+\kappa^{-1}\bigr)^2$};
\node[spec,below=11mm of C] (D) {$\ell(V)>L_m=\tfrac{11}{5}m^2+12m$\\[1pt]
      \scriptsize ($\kappa\ge\kappa_n$, $\operatorname{osc}V\le4n$)};
\node[hess,left=15mm of D] (E) {Stage 1: $\Hess f>b(s)\,g$\\[1pt]
      \scriptsize log barrier, three directions};
\node[hess,left=15mm of E] (F) {Stage 2: $\Hess f>g$\\[1pt]
      \scriptsize balanced constant controls};
\node[box,below=12mm of F,fill=gray!8] (G) {$\lambda_2-\lambda_1>\max\{n+2+\sqrt{n^2+5n+11},\,2n+\tfrac{14}{3}\}$};
\draw[arr] (A) -- node[lab,above]{\cite{Drach}} (B);
\draw[arr] (B) -- node[lab,above]{Lem.~\ref{lem:rationalradial}} (C);
\draw[arr] (C) -- node[lab,right]{Cor.~\ref{cor:radialconversion}} (D);
\draw[arr] (D) -- node[lab,below]{Thm.~\ref{thm:freshspectral}} (E);
\draw[arr] (E) -- node[lab,below]{contact at\\$s<d^2-1$} (F);
\draw[arr] (F) -- node[lab,right,align=left]{Cor.~\ref{cor:oscillator-gap}, \ref{cor:quartic-linear-gap}\\+ diameter bound} (G);
\begin{scope}[on background layer]
\node[draw=green!40!black,dashed,rounded corners,inner sep=6pt,
      fit=(E)(F),label={[font=\scriptsize,anchor=south]above:{continuation along the mean curvature flow family, then along $tV$ (Section~\ref{sec:global})}}] {};
\end{scope}
\end{tikzpicture}
\caption{Structure of the proof of \cref{thm:uniformintro} for $n\ge2$;
for $n=1$ only the constant stage is used (\cref{lem:cp1constant}).
Blue: geometry; orange: spectral inputs; green: the two Hessian
stages, each run by continuation.}
\label{fig:structure}
\end{figure}
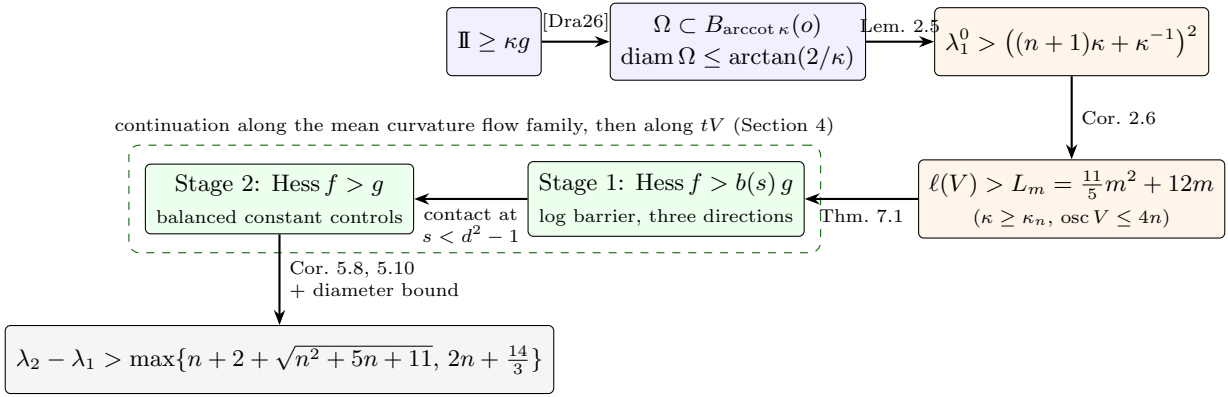

We explain the higher-dimensional obstruction and its resolution;
\cref{fig:structure} records the dependencies. At a first null contact
of $H-bg$, let $v$ be a null direction, $h=Jv$, and
$B=(H-bg)|_{v^\perp}\ge0$. The projective curvature contributes
$-3B(h,h)+3\langle\nabla f,h\rangle^2$. A bound based only on
$\tr B$ loses the distinction between this direction and the transverse
block. Moreover, the estimate $\tr B^2\ge(\tr B)^2/(2n-1)$ weakens
with dimension. Thus a low-dimensional contact estimate does not by
itself give a uniform argument as $n$ increases.

We retain the Hopf Hessian term and test with rotations in both the
longitudinal and transverse directions (Section~\ref{sec:contact}).
The resulting transverse cost is expressed by a trace, without requiring
that $h$ be an eigenvector of $B$ or that $B$ commute with $J$.
These are deterministic tensor tests; no probabilistic representation
is needed. Two radial estimates supply the spectral input. Besides
$\lambda_1^0>n^2\kappa^2$ (Lemma~\ref{lem:smooth-radial-bound}), a
rational radial field on Drach's containing ball gives, by
Lemma~\ref{lem:rationalradial},
\begin{equation}\label{eq:intro-rationalradial}
 \lambda_1^0>\bigl((n+1)\kappa+\kappa^{-1}\bigr)^2
 \qquad(n\ge2,\ \kappa>1),
\end{equation}
which with $\operatorname{osc}V\le4n$ supplies $\ell(V)>L_m$ once
$\kappa\ge\kappa_n$, in particular for $\kappa=\kappa_*$. For the
uniform theorem we use the barrier
\[
 b_{\rm unif}(s)=2\sqrt{1+s}-2d+1-m\log\frac{\sqrt{1+s}+m/2}{d+m/2},
 \qquad d=\frac{9m}{10},\quad m=2n-1.
\]
The contact test retains the longitudinal and Hopf Hessian terms and the
exact trace $T=\mu+s-2nb$. Concavity of the unrestricted trace minimum
(Lemma~\ref{geo:three}) reduces all gradient orientations to the null direction, the Hopf
direction, and their common orthogonal complement. The normalized barrier
$b_{\rm unif}/m$ depends on dimension only through an additive shift, so
the endpoint costs are compared with one limiting profile by differential
inequalities in $1/m$ (Appendix~\ref{app:freshsigns}). A balanced
constant test (Proposition~\ref{prop:balancedcriterion}) then yields Theorem~\ref{thm:uniformintro}.

The barrier is continued along a nested domain family and then along
the potential path $tV$, with $\ell(V)$ controlling every spectral value
(Section~\ref{sec:global}). The gap estimates use a two-point maximum
principle (Proposition~\ref{prop:deterministicmodulus}) with boundary penalty $\varepsilon e^{Mt}(f(x)+f(y)+C)$,
the sharp diameter bound (Proposition~\ref{prop:sharpdiameter})
\begin{equation}\label{eq:diameterintro}
 \II\ge\kappa g,\quad\kappa>1
 \quad\Longrightarrow\quad
 \operatorname{diam}\Omega\le\arctan(2/\kappa),
\end{equation}
attained by geodesic balls, an exact one-dimensional comparison, and a
quartic moment inequality (Propositions~\ref{prop:newscalarcomparison}
and~\ref{prop:quartic-moment-gap}).

The hypothesis $\II\ge\kappa_*g$ forces small domains: by
\eqref{eq:diameterintro}, $D:=\operatorname{diam}\Omega\le\arctan(2/\kappa_*)
<2/\kappa_*<3/4$, since $\kappa_*^2=44/5>64/9$. The proof keeps the projective curvature terms explicitly rather than
absorbing them into an unspecified small-domain error. In particular,
the Hopf contact inequality is checked for every gradient size. The gap bounds are not claimed to be sharp; for comparison, the size $3\pi^2/D^2$ known for convex
spherical domains \cite{SWW,HWZ} satisfies $3\pi^2/D^2>16\pi^2/3$ on such domains. Corollary~\ref{cor:freshuniformgap}
records how the diameter bound improves the Hessian-only gap estimate.

We next relate these results to earlier work. Log-concavity of Dirichlet ground states on convex Euclidean domains
goes back to Brascamp--Lieb \cite{BL}, and lower bounds for the gap
to Singer--Wong--Yau--Yau \cite{SWYY} and Yu--Zhong \cite{YZ}; on
spheres see Lee--Wang \cite{LeeWang}. Two-point and modulus-of-continuity
methods are surveyed by Andrews \cite{AndrewsSurvey}.
Andrews--Clutterbuck \cite{AC} proved the Euclidean fundamental gap
conjecture for convex Schr\"odinger potentials, and Gong--Li--Luo
\cite{GLL} gave a reflection-coupling proof. The spherical estimates
of Seto--Wang--Wei \cite{SWW} and He--Wei--Zhang \cite{HWZ} use
geometric analysis; Cho--Wei--Yang \cite{CWY} use coupling with a
potential modulus. Aryan--Law \cite{AryanLaw} prove concavity
for elliptic and parabolic equations on convex domains in $\CP^n$ and
other locally symmetric spaces of nonnegative curvature; their
results do not cover Dirichlet eigenfunctions. Modified-Hessian continuation is developed in
\cite{KNTW}. In $\CP^2$, Xia \cite{Xia} obtained the Hessian bound $8/5$
and gap $46/5$ for $1$-convex domains by a gradient-dependent barrier and a
reduction of the gradient orientations to three vertices.
Lemma~\ref{geo:three} extends this three-vertex reduction to every
dimension, and our contact test uses finite rotational controls,
including a transverse rotation, in place of his determinant identity.
For $n=2$ his curvature condition and Hessian bound are stronger than
ours; for $n\ge3$, the results below appear to be the first such
estimates without symmetry assumptions. Hsu \cite{Hsu}
establishes strict log-concavity under $U(1)\times U(n-1)$ symmetry in
dimensions two, three, and four; here boundary curvature replaces that
symmetry assumption. The projective curvature split in the distance
equation occurs in Baudoin--Cho--Yang \cite{BCY} and Li--Wang \cite{LW};
explicit scalar estimates are studied by Rutkowski--Seto \cite{RS}.
The diameter bound uses Drach's rolling theorem \cite{Drach} on complex
projective lines.

The paper is organized as follows. Sections~\ref{sec:geometry}--\ref{sec:gap} give spectral inputs and the
diameter bound, the contact identities, continuation, and the passage
from Hessian bounds to gaps. Section~\ref{sec:alldim} proves
Theorem~\ref{thm:schall}; Section~\ref{sec:freshuniform} and
Appendix~\ref{app:freshsigns} prove Theorem~\ref{thm:uniformintro}.

Throughout, $N=2n$ is the real dimension and $m=2n-1$ the real dimension of the
orthogonal complement of a unit vector. We write $f=-\log\phi_1$,
$p=\nabla f$, $s=|p|^2$, $r=\sqrt{1+s}$, $H=\Hess f$, and $\mu$ for the
spectral value entering a contact computation. For a barrier $b(s)$,
$\alpha=b'(s)$, $\beta=b''(s)$, $A=\alpha+2\beta s$, and
$\varkappa=\alpha^2+b\beta$. At a null contact with null vector $v$,
$h=Jv$, $B=(H-bg)|_{v^\perp}$, $T=\tr B$, $p=av+w$ with $w\perp v$,
$\tau=|w|^2$ and $x=\langle w,h\rangle^2$. Letters such as $A,Q,R,W$ are also used for
constants local to a single statement or proof; $\lambda_1^{\rm ND}$
denotes a Neumann--Dirichlet eigenvalue on $(0,D)$.

\section*{AI Disclosure}
In the final stage of preparing this paper, the author used AI tools
(Claude, Anthropic) for two purposes. First, the scalar inequalities
were checked numerically as a consistency test. In the course of these
checks the numerics indicated that a better bound than the one proved
here should hold; since that improvement could not be proved by hand,
the paper does not rely on it. Second, the tools
were used to improve the writing and the figures. The author has checked all
mathematical content and takes full responsibility for it.

\section*{Acknowledgements}
The author is grateful to his postdoctoral mentor, Prof.~Guofang Wei,
for introducing him to this problem in the summer of 2021.

\section{Geometry and spectral inputs}\label{sec:geometry}
Except for Lemma~\ref{lem:sinesquare}(b) and Corollary~\ref{cor:radialconversion}, the eigenvalues in this section are those of the Dirichlet Laplacian; they provide potential-independent spectral allowances for the main theorems.
Here $J$ is the parallel complex structure of $(\CP^n,g)$, with $J^2=-\Id$.
We use the curvature convention $R(U,V)W=\nabla_U\nabla_VW-\nabla_V\nabla_UW-\nabla_{[U,V]}W$. With holomorphic sectional curvature normalized to four (see \cite[Chapter~IX, Sections~6--7]{KN}),
\[
 R(U,V)W=\langle V,W\rangle U-\langle U,W\rangle V+\langle JV,W\rangle JU-\langle JU,W\rangle JV-2\langle JU,V\rangle JW.
\]
Thus $1\le\sec\le4$, $\Ric=(2n+2)g$, $\nabla R=0$, and the injectivity radius is $\pi/2$. For a unit vector $v$ and $h=Jv$, the Jacobi operator is
\begin{equation}\label{eq:jacobi}
 R_v=\Id-v\otimes v+3h\otimes h.
\end{equation}

\begin{lemma}\label{lem:boundaryconnected}
The boundary of a smooth connected $\kappa$-convex domain in $\CP^n$, with $\kappa>0$, is connected.
\end{lemma}
\begin{proof}
If there were more than one boundary component, solve the harmonic Dirichlet problem with different constant values on two components. For the resulting nonconstant harmonic function $w$, all tangential derivatives vanish on the boundary. With the outward normal $\nu_{\rm out}$ and positive outward mean curvature $\mathcal H$, one has $w_{\nu\nu}=-\mathcal H w_\nu$, writing $w_\nu=\partial_{\nu_{\rm out}}w$. The integrated Bochner identity, equivalently the constant-boundary-data case of Reilly's formula \cite{Reilly}, gives
\[
 \int_\Omega\bigl(|\Hess w|^2+\Ric(\nabla w,\nabla w)\bigr)
 =-\int_{\partial\Omega}\mathcal H w_\nu^2\le0.
\]
Positive Ricci curvature forces $w$ to be constant, a contradiction.
\end{proof}
\subsection{A sharp diameter bound}
\label{sec:newgeometry}
\begin{proposition}\label{prop:sharpdiameter}
If $\II\ge\kappa g$ with $\kappa>1$, then
\begin{equation}\label{eq:sharpdiameter}
 \operatorname{diam}\Omega\le D_\kappa:=\arctan(2/\kappa).
\end{equation}
The constant is attained by a geodesic ball of radius $D_\kappa/2$.
\end{proposition}
\begin{proof}
Lemma~\ref{lem:boundaryconnected} supplies connected boundary.
Drach's outer rolling theorem \cite[Theorem~A]{Drach} gives a containing
ball of radius $\arctan(1/\kappa)<\pi/4$; its strict injectivity
condition holds because $\operatorname{inj}(\CP^n)=\pi/2$. Thus $\Omega$
contains no complete complex line, and every interior pair has distance
strictly less than $\pi/2$.

Let $\ell\cong\CP^1=\mathbb S^2(4)$ meet $\partial\Omega$
transversely. A component $E$ of $\Omega\cap\ell$ is a proper smooth
surface domain. If $T$ is its unit boundary tangent, its inward
geodesic curvature is
\[
 k_E(T)=\frac{\II(T,T)}{|\operatorname{proj}_{T\ell}\nu|}
       \ge\kappa.
\]
The numerator identity uses that $\ell$ is totally geodesic.
Gauss--Bonnet gives
$2\pi\chi(E)=4\operatorname{area}(E)+\int_{\partial E}k_E>0$.
A connected proper domain in the sphere has genus zero, so $E$ has
one boundary component and is a disk. Apply the same rolling theorem
inside the curvature-four sphere. Its characteristic radius is
\[
 R_\kappa=\tfrac12\arctan(2/\kappa),\qquad
 2R_\kappa<\pi/2=\operatorname{inj}(\mathbb S^2(4)).
\]
Thus $E$ lies in a spherical cap of radius $R_\kappa$ and has diameter
at most $2R_\kappa$.

Take any distinct $x,y\in\Omega$. Their unique minimizing geodesic
lies in the complex projective line they span, and geodesic convexity
places that segment in $\Omega$. If this line is not transverse to
the boundary, choose a projective unitary isometry $U$ arbitrarily
close to the identity for which $U\ell$ is transverse and $Ux,Uy$
remain in $\Omega$. Such choices exist by parametric transversality:
the evaluation map from the projective unitary group times $\ell$
to $M$ is a submersion. The geodesic from $Ux$ to $Uy$ lies in a
single component of $\Omega\cap U\ell$; its length equals $d(x,y)$.
The preceding cap comparison gives $d(x,y)\le2R_\kappa$.
Taking the supremum proves \eqref{eq:sharpdiameter}.

For a ball of radius $R_\kappa<\pi/4$, the principal curvatures are
$2\cot(2R_\kappa)=\kappa$ in the Hopf direction and
$\cot R_\kappa=(\kappa+\sqrt{\kappa^2+4})/2>\kappa$ in the
remaining directions. The ball is geodesically convex and has
diameter $2R_\kappa$, proving sharpness.
\end{proof}

By a Riccati comparison in the Hopf direction, Xia
\cite[Lemma~4.1]{Xia} showed that the inradius is at most
$\frac12\arctan(2/\kappa)$ and deduced
$\lambda_1^0\ge2\pi^2/\arctan^2(2/\kappa)$;
Proposition~\ref{prop:sharpdiameter} bounds the diameter by twice this
quantity. The caps in different complex lines may have different centers, so the radial estimate below uses the containing radius $\arctan(1/\kappa)$ rather than half the diameter bound.

\subsection{A radial spectral bound}
\begin{lemma}[A projective radial vector-field bound]
\label{lem:smooth-radial-bound}
Suppose $\Omega\subset B_R(o)\subset\mathbb{CP}^n$, with
$R<\pi/2$. Then
\begin{equation}\label{eq:smooth-radial-spectral}
 \lambda_1^0(\Omega)>n^2\cot^2R.
\end{equation}
In particular, under the standing geometric hypotheses,
$\mathrm{II}\ge\kappa g$ with $\kappa>1$ gives
$\lambda_1^0(\Omega)>n^2\kappa^2$.
\end{lemma}
\begin{proof}
Let $\rho=d(o,\cdot)$. The function
$w=-\log\cos\rho$ is smooth on $B_R(o)$, including the center:
$\rho^2$ is smooth below the injectivity radius and
$-\log\cos\sqrt t$ is smooth at $t=0$. In the normalization of
holomorphic sectional curvature four,
\[
 \Delta\rho=(2n-1)\cot\rho-\tan\rho,
 \qquad \Delta w=\sec^2\rho+\tan\rho\,\Delta\rho=2n.
\]
Write $\kappa=\cot R$ and take the smooth vector field
$X=a\nabla w$, $a=n\kappa^2$. It satisfies
$\operatorname{div}X=2na$ and $|X|^2=a^2\tan^2\rho$.
For $u\in H_0^1(\Omega)$, square completion and density give
\[
 \int_\Omega|\nabla u|^2
 =\int_\Omega|\nabla u+Xu|^2
  +\int_\Omega(2na-a^2\tan^2\rho)u^2.
\]
Inside the open ball, $\tan^2\rho<\kappa^{-2}$, so the last
integrand has coefficient strictly greater than
$2na-a^2\kappa^{-2}=n^2\kappa^2$.
Applying this identity to the normalized first Dirichlet eigenfunction
proves the strict inequality; the field is smooth on a neighborhood
of the closed containing ball.
For $\kappa>1$, Drach's theorem \cite[Theorem~A]{Drach}, applied with
sectional lower bound $1$ and connected boundary
(Lemma~\ref{lem:boundaryconnected}), gives the containing ball with
$R=\operatorname{arccot}\kappa$ and $2R<\pi/2$.
\end{proof}

\subsection{Sine-square test functions}
Let $\Omega\subset B_R(o)$, $R<\pi/4$, put $Q=\sin^2R$, and let
$z=\sin^2\rho_o$. The radial Laplacian is
\[
 \Delta F(z)=4\{z(1-z)F''(z)+(n-(n+1)z)F'(z)\}.
\]
\begin{lemma}[Sine-square spectral bounds]\label{lem:sinesquare}
\textup{(a)} If $Q<n/(n+1)$, then $\lambda_1^0(\Omega)\ge4n/Q$. In
particular, under the standing hypotheses, $\II\ge\kappa g$ with
$\kappa>1$ gives
\begin{equation}\label{an:eq:geomspectral}
 \lambda_1^0(\Omega)\ge4n(1+\kappa^2).
\end{equation}
\textup{(b)} If $Q<n/(n+1)$ and, for some $L_0$,
\begin{equation}\label{eq:potentialspectraltest}
 \max V-V(x)\le4\frac{n-(n+1)z(x)}{Q-z(x)}-L_0
 \qquad(x\in\Omega),
\end{equation}
then $\ell(V)\ge L_0$.

\end{lemma}
\begin{proof}
For $\psi=Q-z$ one has
$-\Delta\psi/\psi=4\{n-(n+1)z\}/(Q-z)\ge4n/Q$, since the fraction has
derivative $\{n-(n+1)Q\}/(Q-z)^2>0$. For $u\in C_c^\infty(\Omega)$,
\[
 \int\bigl(|\nabla u|^2+(V-\max V)u^2\bigr)
 =\int\psi^2|\nabla(u/\psi)|^2
  +\int\Bigl(-\frac{\Delta\psi}{\psi}+V-\max V\Bigr)u^2,
\]
and density gives (b), and (a) with $V=0$. Drach's outer rolling
theorem \cite[Theorem~A]{Drach}, with sectional lower bound $1$, gives
a containing ball of radius $R=\operatorname{arccot}\kappa$; the
injectivity condition $2R<\pi/2$ holds, and the boundary is connected
by Lemma~\ref{lem:boundaryconnected}. Then $Q=1/(1+\kappa^2)$.

\end{proof}
Both bounds persist on nested subdomains.

\subsection{A stronger radial comparison}
\begin{lemma}[A rational radial field]\label{lem:rationalradial}
Let $n\ge2$, $\kappa>1$, and $R=\operatorname{arccot}\kappa$.
If $\Omega\subset B_R(o)\subset\mathbb{CP}^n$, with holomorphic
sectional curvature normalized to four, then
\[
 \lambda_1^0(\Omega)>\bigl((n+1)\kappa+\kappa^{-1}\bigr)^2.
\]
\end{lemma}
\begin{proof}
Put $\rho=d(o,\cdot)$, $w=-\log\cos\rho$, $t=\tan^2\rho$,
and $T=\kappa^{-2}$. As in Lemma~\ref{lem:smooth-radial-bound},
$w$ and $t$ are smooth at $o$ and on a neighborhood of the closed
containing ball. They satisfy
\[
 \Delta w=2n,\qquad |\nabla w|^2=t,\qquad
 \nabla t=2(1+t)\nabla w.
\]
Define
\[
 b=\frac{n+1+T}{2nT},\qquad
 a=\frac{(n+1+T)^2}{2nT},\qquad
 X=\frac{a}{1-bt}\nabla w.
\]
Here
\[
 1-bT=\frac{n-1-T}{2n}>0,\qquad
 nb-1=\frac{n+1-T}{2T}>0.
\]
Thus $X$ is smooth and bounded on a neighborhood of the closed
ball; in particular its denominator causes no boundary singularity.
Direct differentiation gives
\[
 \operatorname{div}X-|X|^2
 =\frac{2na}{1-bt}
   +\frac{2ab\,t(1+t)-a^2t}{(1-bt)^2}.
\]
Since $a=b(n+1+T)$ and $nbT=(n+1+T)/2$, subtraction of $2na$
gives the exact identity
\[
 \operatorname{div}X-|X|^2-2na
 =\frac{2ab(nb-1)t(T-t)}{(1-bt)^2}.
\]
It is positive for $0<t<T$. For $u\in H_0^1(\Omega)$, integration
by parts and density give
\[
 \int_\Omega |\nabla u|^2
 =\int_\Omega|\nabla u+Xu|^2
  +\int_\Omega(\operatorname{div}X-|X|^2)u^2.
\]
Apply this to the positive first Dirichlet eigenfunction. The set
$t=0$ consists of at most one point, and every point of $\Omega$
satisfies $t<T$. The last integral is therefore strictly larger than
$2na\int u^2$. Finally,
\[
 2na=\frac{(n+1+T)^2}{T}
      =\bigl((n+1)\kappa+\kappa^{-1}\bigr)^2.
\]
\end{proof}

\begin{corollary}[Explicit spectral conversion]\label{cor:radialconversion}
Under the standing domain assumptions, let $n\ge2$, $\kappa\ge2$,
$\II\ge\kappa g$, and $\operatorname{osc}V\le4n$. Then
\[
 \ell(V)>\bigl((n+1)\kappa+\kappa^{-1}\bigr)^2-4n.
\]
For any $L\ge4$, the sufficient curvature level
\begin{equation}\label{eq:newradialthreshold}
 K_n(L)=\max\left\{2,
 \frac{\sqrt{L+4n}+\sqrt{L-4}}{2(n+1)}\right\}
\end{equation}
therefore implies $\ell(V)>L$.
\end{corollary}
\begin{proof}
Drach's outer rolling theorem \cite[Theorem~A]{Drach}, with
Lemma~\ref{lem:boundaryconnected}, gives $R=\operatorname{arccot}\kappa$.
Apply Lemma~\ref{lem:rationalradial} and
$\ell(V)\ge\lambda_1^0-\operatorname{osc}V$.
The function $(n+1)\kappa+\kappa^{-1}$ increases for $\kappa\ge2$;
its upper quadratic root at the value $\sqrt{L+4n}$ is the second
entry of \eqref{eq:newradialthreshold}.
\end{proof}
\begin{remark}[The leading limit of containing-ball estimates]
\label{rem:leadinglimit}
For each fixed $R\in(0,\pi/2)$,
$\lim_{n\to\infty}n^{-2}\lambda_1^0(B_R\subset\mathbb{CP}^n)=\cot^2R$.
Lemma~\ref{lem:smooth-radial-bound} gives the lower limit. For the upper
limit, the substitution $y=(\sin^{2n-1}r\cos r)^{1/2}u$ turns the radial
quotient into that of $-\partial_r^2+V_n$, where
$V_n(r)=(n-\frac12)(n-\frac32)\csc^2r-\frac14\sec^2r-n^2$. A sine test
function supported in $[R-n^{-1/2},R]$ has kinetic quotient $\pi^2n$,
and $n^{-2}V_n\to\cot^2R$ uniformly there. Thus containment in
$B_{\operatorname{arccot}\kappa}$ cannot supply a spectral target with
quadratic coefficient exceeding $\kappa^2n^2$. Since
$L_m=(11/5)m^2+12m\sim(44/5)n^2$, the level $\kappa_*=\sqrt{44/5}$ of
Theorem~\ref{thm:uniformintro} is asymptotically optimal for this
conversion as $n\to\infty$ (for each finite $n$ the smaller level
$\kappa_n$ of \eqref{eq:intro-newkappa} suffices); it is not claimed
to be necessary for the geometric theorem.
\end{remark}

\section{Rotational tensor tests}\label{sec:contact}
In the contact calculations we specialize \eqref{eq:eigen} to $V=0$;
Section~\ref{sec:schrodinger} adds the potential. Write $N=2n$, $\lambda=\lambda_1$,
$p=\nabla f$, $H=\Hess f$, and $L=\Delta-2\langle p,\nabla\cdot\rangle$.
The ground-state identity is $\phi_1^{-1}(\Delta+\lambda_1)(\phi_1u)=Lu$,
where $\phi_1^{-1}=1/\phi_1$. 
The eigenfunction equation and Bochner formula give
\begin{equation}\label{eq:basic}
 \Delta f=\lambda+s,\qquad \tfrac12Ls=|H|^2+(N+2)s,
 \qquad \nabla s=2Hp.
\end{equation}
For a smooth scalar barrier set $K=H-b(s)g$, $\alpha=b'(s)$,
$\beta=b''(s)$, and define
\begin{equation}\label{eq:tensordriftdefinition}
 L^\nabla K=\sum_i\bigl(\nabla_{e_i}\nabla_{e_i}K
 -\nabla_{\nabla_{e_i}e_i}K\bigr)-2\nabla_pK.
\end{equation}
This differentiates the tensor before evaluation on vectors. For a local
vector field $V$, the scalar product rule instead gives
{\small\[
 L(K(V,V))=(L^\nabla K)(V,V)+4\sum_i(\nabla_{e_i}K)(V,\nabla_{e_i}V)+2\sum_iK(\nabla_{e_i}V,\nabla_{e_i}V)+2K(V,L^\nabla V).
\]}
At a null contact $KV=0$, only the last term vanishes automatically.

\begin{lemma}[Contact identities]\label{lem:contact}
Suppose $K\ge0$ in a neighborhood of a point $x_0$ and $K_{x_0}v=0$ for a unit vector $v$.
All quantities below are evaluated at $x_0$. Decompose the gradient
$p=\nabla f(x_0)$ into its component along $v$ and its orthogonal projection:
\[
 a:=\langle p,v\rangle,\qquad w:=p-\langle p,v\rangle v\in v^\perp.
\]
Thus $p=av+w$ is the orthogonal decomposition of the gradient, and
$s=a^2+|w|^2$. Set
\[
 h=Jv,\quad B=K|_{v^\perp},\quad T=\tr B=\lambda+s-Nb,
 \quad \tau=|w|^2,\quad x=\langle w,h\rangle^2.
\]
Then $B\ge0$ and $0\le x\le\tau\le s$. Writing
$E=\tfrac12(L^\nabla K)(v,v)$, one has
{\small
\begin{equation}\label{eq:Emhigher}
 E=b^2+Nb-\lambda-a^2-3B(h,h)+3x-\alpha\{Nb^2+2bT+\tr B^2+(N+2)s\}-2\beta\{b^2a^2+|(bI+B)w|^2\}.
\end{equation}
}
With $e_0=v$ and $C_j=(\nabla_{e_j}K)(v,\cdot)|_{v^\perp}$,
\begin{equation}\label{eq:Cgeneral}
 C_0=(I+3h\otimes h+2b\alpha I+2\alpha B)w.
\end{equation}
If $e,Je$ is a complex orthonormal pair in $W=\{v,h\}^\perp$ and $q=\langle p,h\rangle$, then
\begin{equation}\label{eq:commutator}
 C_{Je}(e)-C_e(Je)=-2q.
\end{equation}
\end{lemma}
\begin{proof}
Here is an index derivation fixing the curvature signs and the factor of two.
At the point in question use a normal orthonormal frame and set
$\mathcal R_{ikjl}=\langle R(e_i,e_k)e_l,e_j\rangle$; thus
$\mathcal R_{ikjl}v^iv^j$ is the matrix of $R_v$.
The covector commutator, applied successively to $df$, gives
\begin{align*}
 (\Delta^\nabla H)_{ij}
 &= (\Hess\Delta f)_{ij}+\Ric_i{}^kH_{kj}
       +\Ric_j{}^kH_{ik}-2\mathcal R_{ikjl}H_{kl},\\
 (\Hess s)_{ij}
 &=2(H^2)_{ij}+2(\nabla_pH)_{ij}
       +2\mathcal R_{ikjl}p^kp^l.
\end{align*}
There are no curvature-derivative terms because $\nabla R=0$.
Substitute $\Delta f=\lambda+s$ into the first line, then subtract
$2\nabla_pH$ and divide by two. This gives
\[
 \tfrac12(L^\nabla H)(v,v)
 =H^2(v,v)+(N+2)H(v,v)-\langle R_v,H\rangle
       +\langle R(v,p)p,v\rangle.
\]
By \eqref{eq:jacobi}, the curvature contractions are
\[
 \langle R_v,H\rangle=\tr H-H(v,v)+3H(h,h),\qquad
 \langle R(v,p)p,v\rangle=s-a^2+3x.
\]
At contact, $Hv=bv$ and $H|_{v^\perp}=bI+B$. Subtract
\[
 \tfrac12Lb=\alpha\{|H|^2+(N+2)s\}+2\beta|Hp|^2
\]
and use $|H|^2=Nb^2+2bT+\tr B^2$ to obtain \eqref{eq:Emhigher}.

The function $K(V,V)$ has a local minimum for a parallel-to-first-order unit extension $V$ of $v$, so $\nabla_zK(v,v)=0$ for every $z$. The scalar Hessian commutator therefore gives
\[
 (\nabla_vK)(v,z)=2\alpha H(p,z)+\langle R(z,v)v,p\rangle\quad(z\perp v),
\]
which is \eqref{eq:Cgeneral}. For two transverse vectors,
\[
 (\nabla_{e_j}K)(v,e_i)-(\nabla_{e_i}K)(v,e_j)
 =\langle R(e_i,e_j)v,p\rangle.
\]
Taking $(e_i,e_j)=(e,Je)$ and evaluating the displayed curvature tensor yields \eqref{eq:commutator}.
\end{proof}
\begin{remark}[Finite-control cost]\label{rem:finitecontrol}
Let $K\ge0$ near $x_0$ with $K_{x_0}v=0$, let $e_0=v,e_1,\ldots,e_{N-1}$
be an orthonormal basis, and put
\[
 E=\tfrac12(L^\nabla K)_{x_0}(v,v),\qquad
 B=K_{x_0}|_{v^\perp},\qquad
 C_j=(\nabla_{e_j}K)_{x_0}(v,\cdot)|_{v^\perp}.
\]
Identifying the covectors $C_j$ with vectors, define for
$D_0,\ldots,D_{N-1}\in v^\perp$
\begin{equation}\label{eq:CDdef}
 \mathcal C_D=E+2\sum_{j=0}^{N-1}\langle C_j,D_j\rangle
                         +\sum_{j=0}^{N-1}B(D_j,D_j).
\end{equation}
For $B>0$, minimizing over $D$ gives Xia's determinant identity
\cite{Xia}; we use the controls directly.
\end{remark}

\subsection{Longitudinal and transverse rotations}
At a contact as above, $W=\{v,h\}^\perp$ is $J$-invariant.
Choose a complex orthonormal basis $e_1,Je_1,\ldots,e_{n-1},Je_{n-1}$
of $W$; together with $v,h$ it gives a real orthonormal basis.
Index controls by these directions, with $D_0=D_v$ and
$q=\langle p,h\rangle$.
For $A_0\in v^\perp$ and $\sigma\in\mathbb R$, choose
\[
 D_0=-A_0,\quad D_h=0,\quad
 D_{e_\ell}=\sigma Je_\ell,\quad D_{Je_\ell}=-\sigma e_\ell
 \qquad(1\le\ell\le n-1).
\]
Equation~\eqref{eq:commutator} gives the exact cost
\begin{equation}\label{eq:costgeneral}
 \cost=E-2C_0\cdot A_0+B(A_0,A_0)
      +\sigma^2\{T-B(h,h)\}+4(n-1)q\sigma.
\end{equation}
The transverse term is $\sigma^2\operatorname{tr}_W B$; no alignment of $B$ is assumed.

If $A_0=2q_0w$, the coefficient of $B(w,w)$ after expansion is
\begin{equation}\label{eq:cancel}
 4(q_0-\alpha)^2-4\varkappa,\qquad \varkappa=\alpha^2+b\beta.
\end{equation}
It vanishes when $q_0=\alpha+\sqrt\varkappa$ and is nonpositive when $q_0=\alpha$ and $\varkappa\ge0$. In the latter case it may be discarded because the actual matrix $B$ is nonnegative. Only after this step may the remaining quadratic expression be minimized over unrestricted symmetric matrices of fixed trace.

\begin{lemma}[The transverse trace minimum]\label{lem:traceminimum}
Let $\nu\ge2$, $a>0$, $A_\tau=a+2\beta\tau>0$, and $\tau>0$. For a unit vector $h$ and a vector $w$ with $|w|^2=\tau$, write $z=\langle w,h\rangle^2/\tau$. The minimum of
\[
 a\tr B^2+2\beta|Bw|^2+dB(h,h),\qquad B=B^T,\quad\tr B=T,
\]
over real $\nu\times\nu$ matrices is
\[
 P=\frac{\{T+(d/2)(z/A_\tau+(1-z)/a)\}^2}{1/A_\tau+(\nu-1)/a}-\frac{d^2}{4}\left\{\frac{z^2}{A_\tau}+\frac{(1-z)^2}{a}+\frac{4z(1-z)}{a+A_\tau}\right\}.
\]
At $\tau=0$ the minimum is $aT^2/\nu+dT/\nu-d^2(\nu-1)/(4a\nu)$.
\end{lemma}
\begin{proof}
In an adapted orthonormal basis, $w=\sqrt\tau e_1$ and, reversing $h$ if necessary, $h=\sqrt z e_1+\sqrt{1-z}e_2$. The diagonal weights are $A_\tau,a,\ldots,a$, and the only off-diagonal linear term is $2d\sqrt{z(1-z)}B_{12}$. Set
\[
 \ell=\frac{2T+d(z/A_\tau+(1-z)/a)}{1/A_\tau+(\nu-1)/a}.
\]
The candidate has entries
\[
 B^*_{11}=\frac{\ell-dz}{2A_\tau},\quad
 B^*_{22}=\frac{\ell-d(1-z)}{2a},\quad
 B^*_{jj}=\frac\ell{2a}\ (j\ge3),\quad
 B^*_{12}=-\frac{d\sqrt{z(1-z)}}{a+A_\tau},
\]
and all other off-diagonal entries zero. It has trace $T$.
For $C=B-B^*$ the objective minus its value at $B^*$ equals exactly
\[
 A_\tau C_{11}^2+a\sum_{i=2}^{\nu}C_{ii}^2
 +(a+A_\tau)\sum_{j=2}^{\nu}C_{1j}^2
 +2a\sum_{2\le i<j\le\nu}C_{ij}^2.
\]
All linear terms cancel using $\tr C=0$. The displayed weights are positive, proving the minimum, not merely stationarity. Substitution yields $P$. For $\tau=0$ diagonalize $h\otimes h$ and complete the same squares with all diagonal weights equal to $a$; no ratio $z$ is needed. For the actual positive semidefinite contact matrix the minimum is a lower bound.
\end{proof}

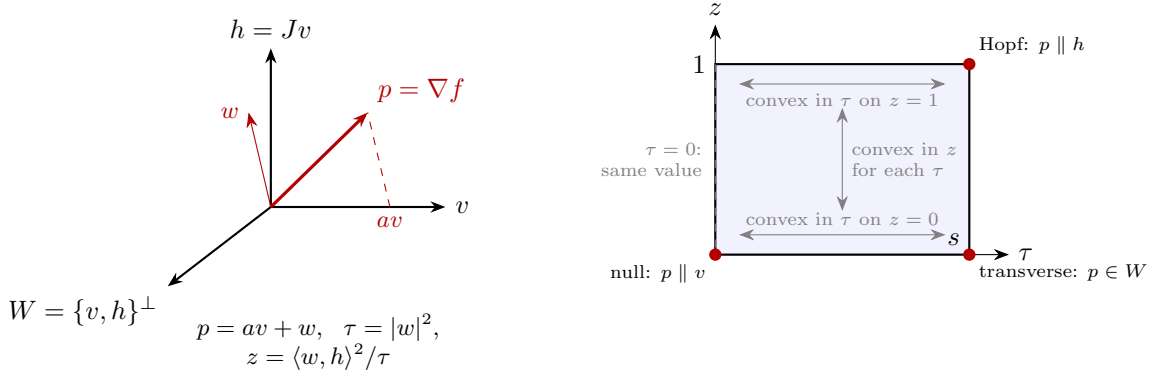
\begin{figure}[t]
\centering
\begin{tikzpicture}[scale=1.05,>={Stealth[length=2mm]}]
 \begin{scope}
  \draw[->,thick] (0,0) -- (2.2,0) node[right]{$v$};
  \draw[->,thick] (0,0) -- (0,2.0) node[above]{$h=Jv$};
  \draw[->,thick] (0,0) -- (-1.3,-1.0) node[below left]{$W=\{v,h\}^\perp$};
  \draw[->,very thick,red!70!black] (0,0) -- (1.22,1.19) node[above right]{$p=\nabla f$};
  \draw[->,red!70!black] (0,0) -- (-0.28,1.19) node[left,font=\small]{$w$};
  \draw[dashed,red!70!black] (1.5,0) -- (1.22,1.19);
  \node[red!70!black,below,font=\small] at (1.5,0) {$av$};
  \node[font=\small,align=center] at (0.6,-1.7)
   {$p=av+w,\ \ \tau=|w|^2,$\\ $z=\langle w,h\rangle^2/\tau$};
 \end{scope}
 \begin{scope}[xshift=5.6cm,yshift=-0.6cm]
  \fill[blue!5] (0,0) rectangle (3.2,2.4);
  \draw[thick] (0,0) rectangle (3.2,2.4);
  \draw[->] (0,0) -- (3.7,0) node[right]{$\tau$};
  \draw[->] (0,0) -- (0,2.9) node[above]{$z$};
  \node[above left] at (3.2,0) {$s$}; \node[left] at (0,2.4) {$1$};
  \foreach \p/\q/\t/\a in {0/0/{null: $p\parallel v$}/{below left},
                           3.2/0/{transverse: $p\in W$}/{below right},
                           3.2/2.4/{Hopf: $p\parallel h$}/{above right}}
    {\fill[red!70!black] (\p,\q) circle (2.2pt);
     \node[font=\scriptsize] at (\p,\q) [\a] {\t};}
  \draw[gray,dashed] (0,2.4) -- (0,0);
  \node[font=\scriptsize,gray,left,align=right] at (-0.05,1.2) {$\tau=0$:\\same value};
  \draw[<->,gray] (0.3,0.25) -- (2.9,0.25) node[midway,above,font=\scriptsize]{convex in $\tau$ on $z=0$};
  \draw[<->,gray] (0.3,2.15) -- (2.9,2.15) node[midway,below,font=\scriptsize]{convex in $\tau$ on $z=1$};
  \draw[<->,gray] (1.6,0.55) -- (1.6,1.85) node[midway,right,font=\scriptsize,align=left]{convex in $z$\\for each $\tau$};
 \end{scope}
\end{tikzpicture}
\caption{The orientation reduction of \cref{geo:three}, extending
Xia's three-vertex argument. Left: the
gradient at a null contact relative to the null vector $v$, the Hopf
vector $h=Jv$, and $W$. Right: the upper contact cost, built from the
unrestricted trace minimum, is convex in $z$ for each $\tau$, and
convex in $\tau$ on the edges $z=0,1$; its maximum over the rectangle
is therefore attained at one of the three marked corners.}
\label{fig:three}
\end{figure}

The following orientation reduction extends the three-vertex argument
of Xia \cite{Xia} from $\CP^2$ to every dimension (\cref{fig:three}).
For unit vectors $e,h\in\mathbb R^\nu$, $\nu\ge2$, write
$z=\langle e,h\rangle^2$. Fix $\alpha>0$, $\beta,\varkappa\in\mathbb R$,
$T>0$, $\eta>0$, and $s\ge0$ with $\alpha+2\beta s>0$, so that
$\alpha+2\beta\tau>0$ for $0\le\tau\le s$. For such $\tau$, let
\[
 P(\tau,z)=\inf_{\substack{B=B^\top\\\tr B=T}}
 \{\alpha\tr B^2+2\beta\tau|Be|^2
                 +4\varkappa\tau B(e,e)+\eta B(h,h)\}.
\]
The matrices in this infimum are not required to be nonnegative.

\begin{lemma}[Three angular endpoints]\label{geo:three}
For any constants $D,c_1,c_2$, the maximum of
$D+c_1\tau+c_2\tau z-P(\tau,z)$ on $[0,s]\times[0,1]$
is attained at one of $(0,0),(s,0),(s,1)$.
\end{lemma}
\begin{proof}
Put $A=\alpha+2\beta\tau$ and $\ell=4\varkappa\tau$.
Choose coordinates with $e=e_1$ and
$h=\sqrt z\,e_1+\sqrt{1-z}\,e_2$.
Completing the diagonal squares under the trace constraint and then
the single mixed square gives
\begin{align}\label{geo:trace-minimum}
 P(\tau,z)={}&
 \frac{\left(T+\frac{\ell+\eta z}{2A}
                    +\frac{\eta(1-z)}{2\alpha}\right)^2}
      {1/A+(\nu-1)/\alpha}
 -\frac{(\ell+\eta z)^2}{4A}
 -\frac{\eta^2(1-z)^2}{4\alpha}
 -\frac{\eta^2z(1-z)}{A+\alpha}.
\end{align}
The diagonal weights are $A,\alpha,\ldots,\alpha$.
The coefficient of $B_{12}^2$ is $A+\alpha$, its linear coefficient
is $2\eta\sqrt{z(1-z)}$, and all other mixed entries minimize at zero.
The coefficient of $z^2$ is
\[
 -\frac{\eta^2(\alpha-A)^2(\nu-2)}
 {4\alpha(A+\alpha)\{A(\nu-1)+\alpha\}}\le0.
\]
Consequently, for each fixed $\tau$, the upper cost is convex in $z$
and is maximized at $z=0$ or $z=1$. Next fix $z\in\{0,1\}$ and hold
$e,h$ fixed.
The objective defining $P$ is affine in $\tau$ for every fixed $B$,
so its infimum over the same trace hyperplane is concave in $\tau$.
The upper cost is convex in $\tau$ and is maximized at $0$ or $s$.
At $\tau=0$ both orientations give the same minimum.
\end{proof}

The infimum is taken without the constraint $B\ge0$: imposing it need not preserve the concavity in $\langle e,h\rangle^2$ used below, and the unrestricted minimum is still a lower bound for the contact matrix.

\subsection{Local scalar tests at a null contact}
\label{sec:localcontacttests}
\begin{lemma}[Finite first-jet contact test]\label{lem:exacttangent}
Use the notation of \cref{rem:finitecontrol}, with $\mathcal C_D$ as in
\eqref{eq:CDdef}. For any prescribed
vectors $D_j\in v^\perp$ there is a smooth local unit vector field
$V$ such that $V(x_0)=v$ and $\nabla_{e_j}V(x_0)=D_j$. At the null
contact, with $F=K(V,V)$,
\begin{equation}\label{eq:deterministiccontactcost}
 \tfrac12 LF(x_0)
 =E+2\sum_j\langle C_j,D_j\rangle+\sum_jB(D_j,D_j)
 =\mathcal C_D\ge0.
\end{equation}
Thus any negative finite-control cost excludes the contact.
\end{lemma}
\begin{proof}
In normal coordinates and a frame parallel to first order at $x_0$,
start with the vector field $v+\sum_j x^jD_j$ and normalize its
length. Since every $D_j$ is perpendicular to $v$, this realizes
the prescribed first derivatives. The scalar function $F$ is
nonnegative near $x_0$ and vanishes there. Hence $LF(x_0)\ge0$.
The product rule displayed at the start of this section gives
\eqref{eq:deterministiccontactcost}: its last term vanishes because
$K(v,\cdot)=0$, while $K(D_j,D_j)=B(D_j,D_j)$.
\end{proof}
For the rotations in \eqref{eq:costgeneral}, this gives exactly
$\mathcal C_D=\cost$. 

\subsection{A longitudinal specialization}\label{sec:commonlongitudinal}
For later use recall $N=2n$, put $m=N-1$, and
$\mathcal B=B$. Choosing $A_0=2\alpha w$, $\sigma=0$ in
\eqref{eq:costgeneral} gives, for the scalar test $F=K(V,V)$,
\begin{align}\label{an:eq:contact}
 \mathcal C_D={}&D_*+[1-4\alpha(1+2b\alpha)]\tau+(3-12\alpha)x\notag\\
 &-\alpha\tr\mathcal B^2-2\beta|\mathcal Bw|^2-3\mathcal B(h,h)
 -4(\alpha^2+b\beta)\mathcal B(w,w),\\
 D_*={}&b^2+Nb-\lambda-s-\alpha\{Nb^2+2bT+(N+2)s\}-2\beta b^2s.\notag
\end{align}

\section{Continuation for domains and potentials}\label{sec:global}
We establish one continuation principle for moving domains and then
for potentials on a fixed domain (Section~\ref{sec:schrodinger}). All uniform bounds below concern a compact
smooth parameter interval, never the singular extinction limit.

\subsection{Domain deformation}\label{sec:strongcontinuation}
\begin{proposition}[A compact contraction family]
\label{prop:strongcontinuation}
Let a smooth connected proper domain $\Omega\subset\CP^n$, $n\ge1$,
have compact closure, nonempty embedded boundary, and $\II\ge\kappa g$,
where $\kappa>0$. There is a nested mean curvature contraction preserving
this curvature bound, whose enclosed domains converge after
rescaling to a Euclidean ball in $C^\infty$. Reversing a compact segment
joins a sufficiently small nearly round member to $\Omega$.
Every lower bound for $\lambda_1^0(\Omega)$ holds along this family.
\end{proposition}
\begin{proof}
The boundary is connected by Lemma~\ref{lem:boundaryconnected}.
For $n\ge2$, the hypersurface dimension is $d=2n-1\ge2$.
Huisken's ambient constants can be taken as $K_1=0$, $K_2=4$, $L=0$:
$1\le\sec\le4$, $\overline\nabla\overline R=0$, and
$\operatorname{inj}(\CP^n)=\pi/2$. His initial tensor inequality
therefore reduces to $\mathcal H\II>0$, where
$\mathcal H=\tr\II$; this holds, since $\II\ge\kappa g$ with $\kappa>0$
gives $\mathcal H\II\ge(2n-1)\kappa^2g>0$. Theorem~1.1 of \cite{Huisken}
gives finite-time extinction at a point and smooth round rescaled
convergence. Embeddedness is preserved by its Lemma~3.2(ii).

The preservation of $\II\ge\kappa g$ requires an additional calculation,
given for $n=2$ in \cite[Lemma~5.1]{Xia}.
With outward normal $\nu_{\rm out}$, our convention is
$\II(X,Y)=\langle\overline\nabla_X\nu_{\rm out},Y\rangle$ and
$\partial_\tau F=-\mathcal H\nu_{\rm out}$.
In a principal orthonormal frame, put $S=g^{-1}\II$, with eigenvalues
$\mu_i$, and $k_{ij}=\sec(e_i,e_j)$. The covariant evolution in
\cite[Lemma~3.3 and Theorem~3.4]{Huisken}, together with
$\partial_\tau g=-2\mathcal H\II$, gives
\begin{equation}\label{eq:shapeaudit}
 ((\partial_\tau-\Delta)S)_i{}^i
 =(|\II|^2+2n+2)\mu_i
       +2\sum_{j\ne i}(\mu_j-\mu_i)k_{ij}.
\end{equation}
Since $\overline\nabla\overline R=0$, the ambient-derivative terms
vanish and the terms $\mathcal H\overline R(\nu,e_i,\nu,e_j)$ from the
time derivative cancel those of Simons' identity. The remaining
tangential curvature contractions are
$-2\mu_i\sum_jk_{ij}+2\sum_j\mu_jk_{ij}$; raising an index
cancels the covariant term $-2\mathcal H\mu_i^2$.
At a first null contact of $S-(\kappa-\varepsilon)I$,
$0<\varepsilon<\kappa$, the chosen $\mu_i$ is least and positive.
Every sectional term in \eqref{eq:shapeaudit} is then nonnegative,
and the first term is positive. The tensor maximum principle \cite{Hamilton} excludes
contact. Letting $\varepsilon\downarrow0$ on each compact flow interval
proves $S\ge\kappa I$.

For $n=1$, use instead the spherical curve-shortening theorem of
Gage--Grayson \cite{Gage,Grayson}, in the formulation and proof of
\cite[Theorem~1.1]{BLZ}, scaled to $\CP^1=\mathbb S^2(4)$.
Its alternatives are smooth round-point extinction and smooth
infinite-time convergence to a great circle. The inward curvature obeys
$(\partial_\tau-\partial_{\ell\ell})k=k^3+4k$;
comparison with $\kappa-\varepsilon$ preserves $k\ge\kappa>0$,
excluding the latter alternative.

Positive inward speed makes the enclosed domains nested. Reverse a
segment $[0,T_0]$ with $T_0<T_*$ chosen near extinction. This reverses
an existing family, not the parabolic equation. Domain monotonicity gives
the spectral assertion. On this compact interval the embeddings and
all required derivatives are bounded. Smooth rescaled convergence follows from
\cite[Theorem~1.1, Lemma~7.2 and p.~479]{Huisken}.
Lemma~7.2 includes the ambient-curvature derivative terms in the
higher-order estimates and extends the interpolation argument of
\cite[Section~8]{Huisken84}. The final step on p.~479 refers to
\cite[Section~10]{Huisken84}: Lemma~10.4 and the ensuing interpolation and Sobolev estimates supply
pointwise derivative bounds, Theorem~10.5 and Lemma~10.6 give curvature and metric decay,
and the conclusion on p.~266 is smooth normalized convergence.
We use this limit only to select one sufficiently small initial member;
all uniform estimates below are on the compact interval $[0,T_0]$.

Smooth convergence of hypersurfaces also supplies the
needed convergence of domains. In rescaled normal coordinates their
boundaries are radial graphs $\rho_\epsilon(\theta)\theta$ with
$\rho_\epsilon\to1$ in every $C^k$ norm. Choose a smooth radial cutoff
$\chi$ equal to zero near $0$ and one near $1$. The maps
\[
 \Psi_\epsilon(z)=
 [1+\chi(|z|)(\rho_\epsilon(z/|z|)-1)]z,
 \qquad \Psi_\epsilon(0)=0,
\]
are diffeomorphisms of the closed unit ball onto the enclosed domains
for small $\epsilon$: their radial derivatives are positive, and they
are the identity near zero. They converge smoothly to the identity.
The rescaled normal-coordinate metrics $g_{ij}(\epsilon z)$ converge
smoothly to the Euclidean metric, hence so do their pullbacks by
$\Psi_\epsilon$. This proves the asserted domain convergence.
\end{proof}

\subsection{A uniform collar and initialization}
\begin{lemma}[Uniform collar]\label{lem:collar}
Let $\Omega_t$, $0\le t\le1$, be a smooth family of connected proper
domains with compact closure and nonempty boundary, parametrized by
embeddings of one compact manifold, and with $\II_t\ge\kappa g>0$.
Let $V_t$ be real smooth potentials depending smoothly on $t$ after
pullback, and $\phi_t$ their positive normalized Dirichlet ground states.
If a smooth barrier satisfies
\[
 b(s)=a\sqrt s-c\log\sqrt s+O(1),\qquad
 0<a\le\kappa,\quad c>0,
\]
then $\Hess(-\log\phi_t)-b(|\nabla\log\phi_t|^2)g$ is positive
in a uniform boundary collar. The same holds for each fixed constant
barrier; the least eigenvalue of $\Hess(-\log\phi_t)$ tends uniformly
to $+\infty$ at the boundary.
\end{lemma}
\begin{proof}
The normal exponential maps have a common tubular radius. Their
differentials are uniformly invertible at distance zero by compactness;
a sequence of failures of injectivity at distances tending to zero would
have coincident limiting base points by embeddedness, contradicting this
local invertibility. Extend the smooth boundary velocity to an ambient
vector field in this collar and then by a cutoff. Its flow gives smooth
diffeomorphisms $\Phi_t:\overline\Omega_0\to\overline\Omega_t$.
The pulled-back metrics are uniformly elliptic with uniformly bounded
derivatives of every fixed order.

Simplicity and normalization give smooth dependence of the eigenpair
on this fixed reference domain. Indeed, writing tildes for pullbacks,
the linearized eigenpair map is
\[
 (v,\ell)\longmapsto
 \left(( -\Delta_{g_t}+\widetilde V_t-\lambda_t)v
             -\ell\widetilde\phi_t,
       2\int\widetilde\phi_t v\,d\mathrm{vol}_{g_t}\right).
\]
It is an isomorphism from
$(C^{k+2,\gamma}\cap H_0^1)\times\mathbb R$ to
$C^{k,\gamma}\times\mathbb R$: pairing its first component $q$
with $\widetilde\phi_t$ determines
$\ell=-\int q\widetilde\phi_t$; the Fredholm alternative determines
the orthogonal component of $v$, and the normalization determines its
remaining component. Boundary Schauder estimates give a bounded inverse
\cite[Chapter~6]{GT}. The implicit function theorem and elliptic
bootstrapping yield smooth local branches in every fixed spatial norm.
Min--max continuity identifies them as ground states, and positivity
and normalization glue them. Compactness makes the bounds uniform.
Only simplicity is used, not the gap estimate to be proved; smooth,
rather than analytic, parameter dependence suffices here (compare
\cite[Chapter~VII]{Kato}).

Write $d$ for boundary distance. Choosing a common
$C\ge\max\{0,\sup(V_t-\lambda_t)\}$ gives
$(\Delta-C)\phi_t\le0$. The boundary point lemma, continuity, and
compactness give a uniform positive inward Hopf derivative. Thus in a
smaller common collar
\[
 \phi_t(y,d)=d\,a_t(y,d),\qquad
 a_t(y,d)=\int_0^1\partial_d\phi_t(y,\theta d)\,d\theta,
\]
where $a_t$ is bounded away from zero and has bounded derivatives.
Consequently
\[
 H_t=\frac{dd\otimes dd}{d^2}-\frac{\Hess d}{d}-\Hess\log a_t,
 \qquad s_t=d^{-2}+O(d^{-1}).
\]
In the normal--tangential splitting, $H_t$ has blocks
$d^{-2}+O(1)$, $\II_t/d+O(1)$, and bounded mixed block.
Since $b(s_t)=a/d+c\log d+O(1)$, absorption of the mixed term gives,
for a uniform $C_1$ and sufficiently small $d$,
\[
 (H_t-b(s_t)g)(\xi\nabla d+Y,\xi\nabla d+Y)
 \ge\frac{\xi^2}{4d^2}
    +\left(\frac{\kappa-a}{d}+c|\log d|-C_1\right)|Y|^2,
 \qquad Y\perp\nabla d.
\]
Here one uses
$2C|\xi||Y|\le\xi^2/(4d^2)+4C^2d^2|Y|^2$.
This proves positivity even when $a=\kappa$, for arbitrary varying
directions. Omitting $b$ proves uniform Hessian divergence and the
constant-barrier assertion.
\end{proof}

\begin{lemma}[Initial positivity]\label{lem:initial}
On a sufficiently small member of the contraction in
Proposition~\ref{prop:strongcontinuation}, the Laplacian ground state
satisfies $H>b(s)g$ for any fixed smooth barrier with
$b(s)\le C(1+\sqrt s)$. One member works for any finite list of barriers.
\end{lemma}
\begin{proof}
For the Euclidean unit-ball ground state in dimension $N=2n$, set
$f_E=-\log\phi_E$ and $h=f_E'/\varrho$. Its radial equation gives
\[
 \varrho h'=\lambda_E-Nh+\varrho^2h^2,
 \qquad h=\frac{\lambda_E}{N}
       +\frac{\lambda_E^2\varrho^2}{N^2(N+2)}+O(\varrho^4).
\]
Thus $h'>0$ initially. A first positive zero of $h'$ would satisfy
$h''=2h^2>0$, a contradiction. Both Hessian eigenvalues $h$ and
$h+\varrho h'$ are positive. On compact interior sets, continuity
therefore gives a positive lower bound. At the boundary the factorization
$\phi_E=d a_E$ gives normal and tangential orders $d^{-2}$ and $d^{-1}$,
whereas $|\nabla f_E|=O(d^{-1})$. Hence
$\Hess f_E\ge c_0(1+|\nabla f_E|)g_E$.

The smooth pullback convergence in
Proposition~\ref{prop:strongcontinuation}, the Rayleigh principle, and
boundary elliptic estimates imply smooth convergence of normalized
first eigenfunctions to $\phi_E$; simplicity and positivity identify
every subsequential limit. In a fixed collar, $C^3$ convergence controls
two derivatives of the positive boundary factors $\phi/d$.
The tangential curvatures have a uniform positive lower bound, so the
normal, tangential, and mixed Hessian blocks are bounded below or above
by $1/(2d^2)$, $c/d$, and a constant, respectively. The same absorption
as in Lemma~\ref{lem:collar}, and $|\nabla f|\le C/d$, give
$\Hess f\ge c_1(1+|\nabla f|)\widehat g$ uniformly in this collar.
Interior convergence supplies this inequality on the complement.
This proves boundary stability as well as interior stability.

With $\widehat g=\epsilon^{-2}g$ and
$z=|\widehat\nabla f|$, the rescaled barrier obeys
$\epsilon^2b(\epsilon^{-2}z^2)\le C\epsilon^2+C\epsilon z$.
For small $\epsilon$ this is less than $c_1(1+z)$.
Scaling back proves the result, simultaneously for finitely many barriers.
\end{proof}

\begin{lemma}[Two ordered continuations]\label{lem:twostageaudit}
Consider a smooth compact domain--potential family covered by
Lemma~\ref{lem:collar}. Suppose the initial member satisfies
$H>b(s)g$ and $H>\eta g$, and both tensors have uniform positive
boundary collars. Assume every nonnegative null contact of the first
tensor admits a finite control with negative cost. Assume the same for
the constant tensor whenever $H>b(s)g$ already holds on the contact
domain. Then both strict inequalities hold throughout the family.
\end{lemma}
\begin{proof}
Use the maps $\Phi_t$ from Lemma~\ref{lem:collar}. Their uniformly
bounded derivatives and inverse derivatives imply
$c d_0\le d_t\circ\Phi_t\le C d_0$ near the boundary.
Choose $\delta>0$ so that $d_0<\delta$ maps into both positive collars,
and put $Q=\{d_0\ge\delta\}\subset\overline\Omega_0$.
If $Q$ is empty there is nothing to prove. Otherwise for either tensor
$K_t$ the function
\[
 m_K(t)=\min_{x\in Q,\ g_0(v,v)=1}
 \frac{(\Phi_t^*K_t)_x(v,v)}{(\Phi_t^*g)_x(v,v)}
\]
is continuous: pulled-back ground states have a positive minimum on
$[0,1]\times Q$, so their logarithmic Hessians are continuous, and the
denominator is uniformly positive. Starting with the first tensor,
a first loss of positivity gives an attained interior null direction
at a parameter $t_*>0$, possibly $t_*=1$. The tensor is nonnegative
on the whole contact domain. A negative finite cost contradicts
Lemma~\ref{lem:exacttangent}. Thus the first inequality holds for all
parameters. Restart the interval for the constant tensor; at its first
contact the first inequality is now available, so the second cost
hypothesis excludes it as well. The parameter is frozen before applying
the local scalar contact test. A fixed domain
with varying potential is included by taking $\Phi_t=\mathrm{Id}$.
\end{proof}

\subsection{The constant stage}\label{sec:cp2balanced}

The constant stage uses the exact trace $T=\lambda+s-N\eta$, $N=2n$.

\begin{proposition}[Balanced-control criterion]\label{prop:balancedcriterion}
Suppose $H>b(s)g$ for an increasing barrier $b$. Fix
$\eta,S,\lambda_*>0$ with $b(S)\ge\eta$ and $\lambda\ge\lambda_*$.
For $0\le c<1$, put
\[
 k=\sqrt{3+c^2},\qquad d_c=8k+4(n-1)c,\qquad
 A=\eta^2+N(1-c^2)\eta-(1-c^2)\lambda_*.
\]
Let $K=H-\eta g$ satisfy $K\ge0$ near an interior point $x_0$ with
$K_{x_0}v=0$, $|v|=1$, and put $h=Jv$, $q=\langle p,h\rangle$, and
\begin{equation}\label{eq:constantcontrols}
 \theta=\begin{cases}1,&q\ge0,\\-1,&q<0,\end{cases}
 \qquad A_0=k\theta h,\qquad\sigma=-c\theta.
\end{equation}
Let $\cost$ be the finite-control cost \eqref{eq:CDdef} of $K$ at
$(x_0,v)$ for the rotation controls of \eqref{eq:costgeneral} with
these $A_0$ and $\sigma$, namely $D_0=-k\theta h$, $D_h=0$,
$D_{e_\ell}=-c\theta Je_\ell$ and $D_{Je_\ell}=c\theta e_\ell$. Then
\begin{equation}\label{eq:balancedcriterion}
 \cost\le\max\{A+c^2S,\ A+(3+c^2)S-d_c\sqrt S\}.
\end{equation}
In particular, if the right side is negative, no such contact exists.
\end{proposition}
\begin{proof}
Use Lemma~\ref{lem:contact}, writing $h=Jv$, $p=av+w$,
$q=\langle p,h\rangle$, $x=q^2$, and
$B=(H-\eta g)|_{v^\perp}\ge0$. At contact the strict first barrier gives $b(s)<\eta\le b(S)$;
monotonicity gives $s<S$, including equality $b(S)=\eta$.
Extend the finite values \eqref{eq:constantcontrols} smoothly in a local orthonormal frame as in
Lemma~\ref{lem:exacttangent}; only these values enter the cost.
Writing $t=B(h,h)$ and $T=\lambda+s-N\eta$, the identities
$k^2-c^2=3$, $\theta^2=1$ and $q\theta=|q|$ give exactly
\begin{equation}\label{eq:balancedexact}
 \cost=\eta^2+N\eta-\lambda-a^2+3x+c^2T-d_c\sqrt x
 \le A+c^2s+3x-d_c\sqrt x.
\end{equation}
Thus the unbounded $t$ term cancels, without requiring $h$ to be an
eigenvector. For $c=0$ and $n=2$, \eqref{eq:balancedexact} reduces, after
discarding $-a^2$, to the constant-stage inequality of Xia
\cite[(6.2)]{Xia}, obtained there by a determinant argument; the
transverse rotation $\sigma=-c\theta$ adds $c^2T-4(n-1)c\sqrt x$ and
replaces $8\sqrt3$ by $8\sqrt{3+c^2}$. For fixed $s$ the last expression is convex in
$x\in[0,s]$. Its $x=0$ value is at most $A+c^2S$; its $x=s$
value is convex in $s\in[0,S]$, with the zero endpoint no larger
than $A+c^2S$. This proves \eqref{eq:balancedcriterion}.
Lemma~\ref{lem:exacttangent} excludes a negative value of this contact cost.
\end{proof}

\subsection{Potentials on a fixed domain}\label{sec:schrodinger}
We use the Laplacian estimates to initialize the path $-\Delta+tV$ on
one fixed domain. Write $\lambda_t=\lambda_1(-\Delta+tV)$,
$\mu_t=\lambda_t-tV$, and
$\ell(V)=\lambda_1(-\Delta+V)-\max_{\overline\Omega}V$.

\begin{lemma}[Endpoint spectral control]\label{lem:endpointspectral}
One has
\begin{equation}\label{eq:potentialmu}
 \inf_{0\le t\le1,\ x\in\Omega}\mu_t(x)=\ell(V).
\end{equation}
Thus $\ell(V)\ge L_0$ gives $\mu_t\ge L_0$ along the whole path,
including $\lambda_1^0\ge L_0$ at $t=0$.
The condition $\lambda_1^0-\operatorname{osc}_\Omega V\ge L_0$ suffices.
\end{lemma}
\begin{proof}
Put $W=V-\max V\le0$. Since $tW\ge W$, the Rayleigh principle gives
$\mu_t(x)\ge\lambda_1(-\Delta+tW)\ge\lambda_1(-\Delta+W)=\ell(V)$.
At $t=1$, approach a maximum point of $V$ to obtain equality of the
infima. The last assertion follows from
$\lambda_1(-\Delta+V)\ge\lambda_1^0+\min V$.
\end{proof}

When $\Omega\subset B_R(o)$ with $R<\pi/4$ (for instance
$R=\operatorname{arccot}\kappa$ under $\II\ge\kappa g$, $\kappa>1$),
Lemma~\ref{lem:sinesquare}(b) gives a pointwise sufficient condition for $\ell(V)\ge L_0$.

For the normalized ground states set
$f_t=-\log\phi_t$, $p_t=\nabla f_t$, $s_t=|p_t|^2$,
$H_t=\Hess f_t$, and $L_t=\Delta-2p_t\cdot\nabla$.
The eigenfunction equation and Bochner identity give, with $N=2n$,
\[
 \Delta f_t=\mu_t+s_t,\qquad
 \tfrac12L_ts_t=|H_t|^2+(N+2)s_t-t\langle p_t,\nabla V\rangle.
\]
At a nonnegative null contact of $H_t-b(s_t)g$, tensor commutation adds
$-t\Hess V/2$ to the contact cost of $H_t$. Thus the controlled cost is
\begin{equation}\label{eq:potentialcost}
 \cost_t=\cost^{(0)}(\mu_t)
       -\tfrac t2\Hess V(v,v)
       +t b'(s_t)\langle p_t,\nabla V\rangle.
\end{equation}
Here $\cost^{(0)}(\mu_t)$ is the cost of Section~\ref{sec:contact}
with the actual contact data and scalar $\mu_t(x_0)$:
$B=(H_t-b(s_t)g)|_{v^\perp}\ge0$ and
$\tr B=\mu_t+s_t-Nb$. The first-derivative contact identities receive
no potential correction.

\begin{proposition}[Fixed-domain two-stage continuation]
\label{prop:fixedpotentialcontinuation}
Let $\II\ge\kappa g>0$, and let a smooth barrier $b$ satisfy
$b'\ge0$,
\[
 b(s)=a\sqrt s-c\log\sqrt s+O(1),\qquad
 0<a\le\kappa,\quad c>0,\qquad
 c_b:=\sup_{s\ge0}2\sqrt s\,b'(s)<\infty.
\]
Suppose $H_0>b(s_0)g$ and $H_0>\eta g$. For every scalar
$\mu\ge L_0$, suppose the contact algebra gives a strictly negative
finite cost at each nonnegative null contact of the first barrier,
and at each constant-barrier contact satisfying $b(s)<\eta$.
If
\begin{equation}\label{eq:potentialconditions}
 \ell(V)\ge L_0,\qquad \Hess V\ge c_b|\nabla V|g,
\end{equation}
then $H_t>b(s_t)g$ and $H_t>\eta g$ for all $0\le t\le1$.
\end{proposition}
\begin{proof}
Lemma~\ref{lem:collar}, with fixed domain and potential $tV$, supplies
smooth eigenpair dependence and uniform collars for both tensors,
including signed $V$ and $a=\kappa$. Lemma~\ref{lem:endpointspectral}
gives $\mu_t\ge L_0$. At the first-barrier contact,
\begin{equation}\label{eq:potentialcorrectionsign}
 -\tfrac t2\Hess V(v,v)+t b'(s_t)\langle p_t,\nabla V\rangle
 \le t\{-c_b/2+b'(s_t)\sqrt{s_t}\}|\nabla V|\le0.
\end{equation}
The first contact cost remains negative. Once the first inequality
holds for the entire path, a constant-barrier contact satisfies
$b(s_t)<\eta$ and adds only $-t\Hess V(v,v)/2\le0$.
Lemma~\ref{lem:twostageaudit} proves the two conclusions in that order.
\end{proof}

\begin{proposition}[Spending a strict contact margin]
\label{prop:potentialmargin}
Retain the domain, barrier, collar, and initialization hypotheses of
Proposition~\ref{prop:fixedpotentialcontinuation}. Suppose its two
algebraic costs can be chosen strictly below $-M_b$ and $-M_\eta$,
respectively, for every $\mu\ge L_0$ and every relevant contact datum,
where $M_b,M_\eta>0$. Its conclusions remain valid under
\begin{equation}\label{eq:potentialmargins}
 \ell(V)\ge L_0,\qquad
 \Hess V\ge\max\{c_b|\nabla V|-2M_b,-2M_\eta\}\,g.
\end{equation}
\end{proposition}
\begin{proof}
Put $\chi_V=\min_{|v|=1}\Hess V(v,v)$.
The first-stage correction is at most
$t(c_b|\nabla V|-\chi_V)/2\le M_b$, and the second is at most
$-t\chi_V/2\le M_\eta$. These estimates include the cases where the
expressions before multiplication by $t$ are negative.
Both total costs remain strictly negative, including equality in
\eqref{eq:potentialmargins}. The collar uses only smoothness of $V$;
Lemma~\ref{lem:twostageaudit} therefore applies.
\end{proof}

\section{Hessian bounds and fundamental gaps}\label{sec:gap}
The passage from a lower bound for $\Hess(-\log\phi_1)$ to a lower
bound for $\lambda_2-\lambda_1$ is classical. In the Euclidean case it
goes back to Singer--Wong--Yau--Yau \cite{SWYY} and Yu--Zhong
\cite{YZ}, and the sharp estimate of Andrews--Clutterbuck \cite{AC}
rests on a two-point maximum principle for the eigenfunction quotient;
see also the survey \cite{AndrewsSurvey} and the comparison survey of
Dai--Seto--Wei \cite{DSW19}. On spheres, sharp gap estimates were
proved by Seto--Wang--Wei \cite{SWW}, He--Wei--Zhang \cite{HWZ} and
Dai--Seto--Wei \cite{DSW21}, and on surfaces of positive curvature
log-concavity and gap estimates were obtained by
Khan--Nguyen--Tuerkoen--Wei \cite{KNTW}; Gong--Li--Luo \cite{GLL} and
Cho--Wei--Yang \cite{CWY} give coupling proofs. We include the
arguments because the curvature of $\CP^n$ splits into a Hopf part and
a transverse part, and because the potential and the Hessian bound
enter only through the drift.
\subsection{Weighted Bochner transfer}
\begin{proposition}[Hessian-to-gap estimate]\label{prop:commonbochner}
For the Schr\"odinger eigenpairs in \eqref{eq:eigen}, $H=\Hess(-\log\phi_1)>\eta g$ implies $\lambda_2-\lambda_1>2n+2+2\eta$.
\end{proposition}
\begin{proof}
Put $u=\phi_2/\phi_1$ and $\gamma=\lambda_2-\lambda_1$. Factoring $\phi_i=\delta a_i$ at the boundary, with $a_1>0$, shows that $u$ is smooth on $\overline\Omega$. The ground-state identity gives $Lu=-\gamma u$, so weighted Bochner gives
\[
 \tfrac12L|\nabla u|^2=|\Hess u|^2-\gamma|\nabla u|^2+
 ((2n+2)g+2H)(\nabla u,\nabla u).
\]
Integrate against $\phi_1^2$ on $\{\delta>\varepsilon\}$. The flux is $O(\varepsilon^2)$, while $H=O(\delta^{-2})$ and $\phi_1^2=O(\delta^2)$ make every term integrable. Letting $\varepsilon\downarrow0$ gives
\[
 \gamma\int|\nabla u|^2\phi_1^2
 =\int\{|\Hess u|^2+((2n+2)g+2H)(\nabla u,\nabla u)\}\phi_1^2
 >(2n+2+2\eta)\int|\nabla u|^2\phi_1^2.
\]
The last integral is positive since $u$ is nonconstant.
\end{proof}
\cref{prop:commonbochner} is the Lichnerowicz estimate for the drift
Laplacian $L$, whose Bakry--\'Emery Ricci tensor is
$\Ric+2\Hess f$ \cite{BakryEmery}; Xia uses the same argument in
$\CP^2$ \cite[Corollary~1.4]{Xia}.

\subsection{A deterministic two-point comparison}
\label{sec:twopoint}
Throughout this subsection the standing geodesic convexity assumption
is retained, and the domain and potential are fixed. Assume
\[
 H=\Hess(-\log\phi_1)\ge\eta g,\qquad \eta\ge0,
 \qquad \operatorname{diam}\Omega\le D<\pi/2.
\]
Set $f=-\log\phi_1$, $p=\nabla f$, and
$L=\Delta-2\langle p,\nabla\cdot\rangle$. Define
\begin{equation}\label{eq:deterministicdistanceoperator}
 B_{n,\eta}(r)=4(n-1)\tan(r/2)+4\tan r+2\eta r,
 \qquad \mathcal A_{n,\eta}=4\partial_{rr}-B_{n,\eta}\partial_r.
\end{equation}
The proof below follows the two-point maximum principle of
Andrews--Clutterbuck \cite{AC} (see also \cite{AndrewsSurvey}). The
index form with Hopf and transverse curvatures in
\cref{lem:twopointoperator} is the curvature splitting used by
Baudoin--Cho--Yang \cite{BCY} and Li--Wang \cite{LW}; the singular
ground-state drift is handled here by an explicit interior barrier.

\begin{lemma}[Boundary regularity of the eigenfunction quotients]
\label{lem:quotientboundary}
For every Dirichlet eigenfunction $\phi_j$, the quotient
$u_j=\phi_j/\phi_1$ is smooth on $\overline\Omega$ and satisfies
\[
 \partial_\nu u_j=0\quad\hbox{on }\partial\Omega,
 \qquad Lu_j=-(\lambda_j-\lambda_1)u_j\quad\hbox{in }\Omega.
\]
Moreover $f\to+\infty$ at the boundary and
\begin{equation}\label{eq:logboundarybarrier}
 Lf=\lambda_1-V-|\nabla f|^2\le\lambda_1-\min_{\overline\Omega}V.
\end{equation}
\end{lemma}
\begin{proof}
This is standard. In a boundary collar let $d$ denote inward boundary
distance. Boundary Schauder regularity \cite[Chapter~6]{GT} and the
Hopf lemma \cite[Lemma~3.4]{GT} give $\phi_j=d a_j$, with
$a_j$ smooth and $a_1>0$ up to the boundary. Thus $u_j=a_j/a_1$ is
smooth. On $d=0$, the eigenfunction equation gives $\Delta\phi_j=0$,
and the product rule gives
\[
 0=\Delta(d a_j)=2\partial_d a_j+a_j\Delta d.
\]
The same identity for $j=1$ shows $\partial_d(a_j/a_1)=0$.
The ground-state identity proves the equation for $u_j$.
Finally, $f=-\log d-\log a_1$ diverges to $+\infty$ and
$\Delta f=\lambda_1-V+|\nabla f|^2$, proving
\eqref{eq:logboundarybarrier}.
\end{proof}

\begin{lemma}[Two-point differential inequality]
\label{lem:twopointoperator}
For distinct interior points $x,y$, write $r=d(x,y)$ and let
$\gamma:[0,r]\to\Omega$ be their unit-speed minimizing geodesic.
There is a positive semidefinite second-order operator $\mathscr P$
at $(x,y)$ such that, for smooth $v$ and $\psi$,
\begin{align}
 \mathscr P\{v(y)-v(x)\}&=Lv(y)-Lv(x),
 \label{eq:pairmarginalidentity}\\
 \mathscr P\{\psi(r)\}
 &=4\psi''(r)-\left\{4(n-1)\tan(r/2)+4\tan r
       +2\int_0^rH(\dot\gamma,\dot\gamma)\,ds\right\}\psi'(r).
 \label{eq:pairdistanceidentity}
\end{align}
Consequently $\mathscr P\{\psi(r)\}\le\mathcal A_{n,\eta}\psi(r)$
whenever $\psi'\ge0$. Also
$\mathscr P\{f(x)+f(y)\}=Lf(x)+Lf(y)$.
\end{lemma}
\begin{proof}
Since $r<\pi/2$, the distance is smooth at $(x,y)$ and the geodesic
is unique. Choose an orthonormal basis $e_0=\dot\gamma(0)$,
$e_1=Je_0,e_2,\ldots,e_{2n-1}$, with the last $2n-2$ vectors
orthogonal to $\{e_0,Je_0\}$, and let $P$ denote parallel transport
from $x$ to $y$. In the product tangent space put
\[
 E_0=(-e_0,Pe_0),\qquad E_i=(e_i,Pe_i)\quad(i\ge1),
 \qquad Z=(-2p(x),-2p(y)).
\]
For a smooth function $F$ on the product define, pointwise,
\[
 \mathscr PF=\sum_{i=0}^{2n-1}\Hess_{\Omega\times\Omega}F(E_i,E_i)
                 +dF(Z).
\]
The definition does not require differentiating the selected frame.
Its principal form is positive semidefinite, and each diagonal block
is the metric, which proves \eqref{eq:pairmarginalidentity} and the
identity for $f(x)+f(y)$.

One has $dr(E_0)=2$, $\Hess r(E_0,E_0)=0$, and $dr(E_i)=0$ for
$i\ge1$. For a parallel unit transverse field $e(s)$ of curvature
$k$, the Jacobi field with equal parallel endpoint values is
\[
 J_k(s)=\frac{\cos(\sqrt k(s-r/2))}{\cos(\sqrt k r/2)}e(s).
\]
Its index form is the endpoint term
$[\langle J_k',J_k\rangle]_0^r=-2\sqrt k\tan(\sqrt k r/2)$.
The second variation of distance therefore gives this value for
$\Hess r((e(0),e(r)),(e(0),e(r)))$. There is one Hopf direction
with $k=4$ and $2n-2$ directions with $k=1$; the latter term is
absent when $n=1$. Finally,
\[
 dr(Z)=-2\{\langle p(y),\dot\gamma(r)\rangle
                  -\langle p(x),\dot\gamma(0)\rangle\}
        =-2\int_0^r H(\dot\gamma,\dot\gamma)\,ds.
\]
These identities prove \eqref{eq:pairdistanceidentity}, including
the coefficient $4$ of $\psi''$.
\end{proof}

\begin{proposition}[Deterministic modulus contraction]
\label{prop:deterministicmodulus}
Let $\psi\in C([0,D])\cap C^2((0,D))$ satisfy
$\psi(0)=0$, $\psi>0$ on $(0,D]$, $\psi'\ge0$ on $(0,D)$, and
\[
 \mathcal A_{n,\eta}\psi\le-\Lambda\psi\quad(0<r<D).
\]
Suppose $v$ is continuous on $\overline\Omega\times[0,T]$, is a
classical solution of $\partial_t v=Lv$ in $\Omega\times(0,T]$,
and, for some $C\ge0$, satisfies
$|v(y,0)-v(x,0)|\le C\psi(d(x,y))$. Then
\begin{equation}\label{eq:deterministiccontraction}
 |v(y,t)-v(x,t)|\le C e^{-\Lambda t}\psi(d(x,y))
 \quad(x,y\in\overline\Omega,\ 0\le t\le T).
\end{equation}
No positive lower bound for $\psi'(D)$ is required.
\end{proposition}
\begin{proof}
Put $f_*=f-\min_\Omega f+1/2$ and
$\Phi(x,y)=f_*(x)+f_*(y)\ge1$. Choose
\[
 M=1+\max\{0,2(\lambda_1-\min_{\overline\Omega}V)\}.
\]
For $\varepsilon>0$ consider
\[
 Z_\varepsilon(x,y,t)=v(y,t)-v(x,t)
       -Ce^{-\Lambda t}\psi(d(x,y))-\varepsilon e^{Mt}\Phi(x,y).
\]
At $t=0$ this is negative. It tends uniformly to $-\infty$ as
either endpoint approaches the boundary, because $v$ is bounded on
the finite time slab and $f\to+\infty$ (Lemma~\ref{lem:quotientboundary}). On the diagonal it equals
$-\varepsilon e^{Mt}\Phi<0$. Thus, if it is positive somewhere,
its positive maximum is attained at $(x_0,y_0,t_0)$ with
$x_0,y_0\in\Omega$, $x_0\ne y_0$, and $t_0>0$.

Every distinct interior pair has $d(x,y)<D$: if equality held,
extend its minimizing geodesic slightly past $y$ inside the open
domain; the extended segment, still shorter than $\pi/2$, would
contradict the diameter bound. Hence the distance and $\psi(d)$
are smooth near the maximum, with no cut-locus or comparison-endpoint
contact. At this maximum $\partial_t Z_\varepsilon\ge0$,
$\nabla Z_\varepsilon=0$, and $\Hess Z_\varepsilon\le0$.
These statements remain valid at $t_0=T$ using the left time
derivative. Lemma~\ref{lem:twopointoperator} therefore gives
$(\partial_t-\mathscr P)Z_\varepsilon\ge0$. On the other hand,
the equation for $v$, the scalar inequality, and
\eqref{eq:logboundarybarrier} imply
\begin{align*}
 (\partial_t-\mathscr P)Z_\varepsilon
 &\le -\varepsilon e^{Mt}\{M\Phi-\mathscr P\Phi\}\\
 &\le-\varepsilon e^{Mt}<0,
\end{align*}
a contradiction. Thus $Z_\varepsilon\le0$. For each fixed
interior pair let $\varepsilon\downarrow0$, exchange $x,y$, and
then pass to boundary endpoints by continuity. This proves the
assertion; no boundary maximum principle for the singular operator
$L$ is used.
\end{proof}

\begin{corollary}[Scalar costs give spectral gaps]
\label{cor:deterministiccostgap}
If, in addition, $\psi(r)\ge cr$ on $[0,D]$ for some $c>0$,
then $\lambda_2-\lambda_1\ge\Lambda$.
\end{corollary}
\begin{proof}
By Lemma~\ref{lem:quotientboundary}, the quotient $u=\phi_2/\phi_1$ is
smooth, hence Lipschitz, on the closure.
Geodesic convexity gives
$|u(y)-u(x)|\le\|\nabla u\|_\infty d(x,y)\le C\psi(d(x,y))$.
Apply Proposition~\ref{prop:deterministicmodulus} to
$v(x,t)=e^{-(\lambda_2-\lambda_1)t}u(x)$ for arbitrary finite $T$.
Choose an interior pair distinguished by the nonconstant $u$ and
compare the exponential rates as $t\to\infty$.
\end{proof}

\subsection{Exact scalar comparison and spectral variance}
\label{sec:newscalar}
Write $\gamma=\lambda _2-\lambda _1$ and suppose that
$\operatorname{Hess}(-\log\phi _1)\ge\eta g$, $\eta\ge0$, and
$\operatorname{diam}\Omega\le D<\pi/2$. Set
\[
 B(r)=4(n-1)\tan(r/2)+4\tan r+2\eta r,
 \qquad W(r)=\frac{B(r)^2}{16}+\frac{B'(r)}2.
\]
The holomorphic and orthogonal curvature splitting was derived in
Lemma~\ref{lem:twopointoperator}; all spectral transfers below are
deterministic. Reducing the gap to a one-dimensional eigenvalue
problem is the step of Andrews--Clutterbuck \cite{AC};
\cref{prop:newscalarcomparison} records the exact one-dimensional
comparison in the present curvature setting.
All the bounds in this section apply to the Schr\"odinger ground state:
the potential enters through the independently established Hessian bound.

\begin{proposition}[The exact one-dimensional comparison]
\label{prop:newscalarcomparison}
Let $\zeta$ be the first eigenvalue of
\[
 -4z''+Wz=\zeta z,\qquad z'(0)=0,\quad z(D)=0.
\]
Then $\gamma\ge\zeta$. Consequently every pointwise lower bound
$W>Q$ on $(0,D)$ yields
$\gamma>\lambda_1^{\rm ND}(-4\partial_{rr}+Q;D)$.
\end{proposition}
\begin{proof}
Let $z$ be the positive first eigenfunction of
$-4z''+Wz=\zeta z$, with $z'(0)=0$, $z(D)=0$, where
$W=B^2/16+B'/2$ and $B=B_{n,\eta}$. Set
\[
 F(r)=\frac18\int_0^r B(t)\,dt,\qquad
 w=e^Fz,\qquad \psi(r)=\int_0^r w(t)\,dt.
\]
Direct differentiation gives
\[
 -4w''+Bw'+B'w=\zeta w,
 \qquad w'(0)=0,
 \qquad 4\psi''-B\psi'=-\zeta\psi.
\]
For the last identity, differentiate its left side plus $\zeta\psi$
and use $B(0)=\psi(0)=w'(0)=0$ to determine the integration
constant. The function $\psi$ is smooth on $[0,D]$, is positive and strictly
increasing before $D$, and satisfies $\psi'(D)=0$ and
$\min_{[0,D]}\psi(r)/r>0$, where the ratio at zero is $w(0)>0$.
Corollary~\ref{cor:deterministiccostgap} gives $\gamma\ge\zeta$.
If $W>Q$ on $(0,D)$, evaluate the two quadratic forms on the
positive first eigenfunction $z$ for $W$:
\[
 \zeta=\frac{\int_0^D(4z'^2+Qz^2)}{\int_0^D z^2}
       +\frac{\int_0^D(W-Q)z^2}{\int_0^D z^2}
       >\lambda_1^{\rm ND}(-4\partial_{rr}+Q;D).
\]
This proves the stated strict comparison.
\end{proof}

\subsection{Diameter-sensitive costs}\label{sec:diametercost}
\begin{proposition}[Diameter and Hessian gap bound]\label{prop:diametergap}
If $n\ge1$, $\operatorname{diam}\Omega\le D<\pi/2$, and $H\ge\eta g$ with $\eta\ge0$, put
$a=n+1+\eta$, $C=(a^2+n+7)/4$, and $P=\pi^2/D^2$. Then
\begin{equation}\label{eq:diametergap}
 \lambda_2-\lambda_1>
 a+2\sqrt C+\frac{(P-\sqrt C)_+^2}{P},
 \qquad x_+=\max\{x,0\}.
\end{equation}
\end{proposition}
\begin{proof}
By Proposition~\ref{prop:newscalarcomparison}, it suffices to bound
the first Neumann--Dirichlet eigenvalue with potential $W$.
Positivity of the tangent-series coefficients gives, for $r>0$,
$B(r)\ge2ar+(n+7)r^3/6$ and $B^2/16+B'/2>a+Cr^2$.
For any form-domain function $z$ with $z(D)=0$, Neumann--Dirichlet Wirtinger gives $4\int z'^2\ge P\int z^2$, while
\[
 0\le\int_0^D(2\sqrt\varepsilon z'+\sqrt C\,rz)^2
 =4\varepsilon\int z'^2+C\int r^2z^2-2\sqrt{\varepsilon C}\int z^2
 \quad(0\le\varepsilon\le1).
\]
The boundary term $[rz^2]_0^D$ vanishes. Splitting the kinetic term yields
$(4\int z'^2+C\int r^2z^2)/\int z^2\ge(1-\varepsilon)P+2\sqrt{\varepsilon C}$.
Choose $\varepsilon=C/P^2$ if $P\ge\sqrt C$, otherwise $\varepsilon=1$, and use the strict potential comparison on the positive first eigenfunction.
\end{proof}

\begin{corollary}[A dimension-dependent improvement of the Bochner bound]
\label{cor:oscillator-gap}
Under the standing smooth geodesically convex domain hypotheses, if
$\II\ge\kappa g$ for $\kappa>1$ and $H\ge g$, then
\[
 \lambda_2-\lambda_1>
 n+2+\sqrt{n^2+5n+11}>2n+\frac92.
\]
\end{corollary}
\begin{proof}
Proposition~\ref{prop:sharpdiameter} gives
$\operatorname{diam}\Omega\le\arctan(2/\kappa)=:D<\pi/2$.
In Proposition~\ref{prop:diametergap} take $\eta=1$ and discard
its nonnegative last term. The final strict comparison follows from
$n^2+5n+11>(n+5/2)^2$.
\end{proof}

\begin{proposition}[A quartic moment gap bound]
\label{prop:quartic-moment-gap}
Under the standing geodesically convex domain hypotheses, suppose
$\operatorname{diam}\Omega\le D<\pi/2$ and
$\Hess(-\log\phi_1)\ge\eta g$, where $\eta\ge0$. Put
\[
 a=n+1+\eta,\qquad C=\frac{a^2+n+7}{4},\qquad
 E=\frac{a(n+7)+n+31}{24},\qquad P=\frac{\pi^2}{D^2}.
\]
Let $y_0>0$ be the unique solution of
$2Ey_0^3+Cy_0^2=1$, and set $y_*=\min\{y_0,P^{-1}\}$.
Then
\[
 \lambda_2-\lambda_1>
 a+\frac1{y_*}+Cy_*+Ey_*^2.
\]
This strictly improves the right side of
\eqref{eq:diametergap} for the same $n,\eta,D$.
\end{proposition}
\begin{proof}
The positive Taylor coefficients of the tangent function give
\[
 B(r)>2ar+\frac{n+7}{6}r^3+\frac{n+31}{60}r^5
 \quad(0<r<D).
\]
Using the corresponding derivative inequality in
$W=B^2/16+B'/2$ yields $W(r)>a+Cr^2+Er^4$.
For a normalized function $z\in H^1(0,D)$ with $z(D)=0$, put
$y=\int_0^D r^2z^2>0$. Integration by parts and
Cauchy--Schwarz imply
\[
 1=-2\int_0^D rzz',\qquad 4\int_0^D z'^2\ge\frac1y.
\]
Neumann--Dirichlet Wirtinger and Cauchy--Schwarz also give
\[
 4\int_0^D z'^2\ge P,\qquad
 \int_0^D r^4z^2\ge y^2.
\]
Consequently Proposition~\ref{prop:newscalarcomparison} gives
\[
 \lambda_2-\lambda_1>
 a+\min_{y>0}\left\{\max\left(P,\frac1y\right)+Cy+Ey^2\right\}.
\]
For $y\ge P^{-1}$ the minimized expression is increasing. On
$(0,P^{-1}]$ its derivative has the sign of $2Ey^3+Cy^2-1$,
proving the formula. Omitting $Ey^2$ gives exactly the previous
oscillator bound. Both minima occur at finite positive arguments,
so the positive term makes the present minimum strictly larger.
\end{proof}

\begin{corollary}[A stronger dimension-linear gap]
\label{cor:quartic-linear-gap}
Under the hypotheses of Corollary~\ref{cor:oscillator-gap}, $\Hess(-\log\phi_1)\ge g$ implies
\[
 \lambda_2-\lambda_1>2n+\frac{14}{3}.
\]
\end{corollary}
\begin{proof}
Here $a=n+2$, $C=(n^2+5n+11)/4$, and
$E=(n^2+10n+45)/24$. We claim that, for every $y>0$,
\[
 \frac1y+Cy+Ey^2>n+\frac83.
\]
Multiply the difference by $y$ and put $q=ny/2$. The result is
\[
 (1-q)^2+\frac{q(q-1)(q+16)}{3n}
 +\frac{q^2(33+10q)}{3n^2}+\frac{15q^3}{n^3}.
\]
It is positive for $q\ge1$. For $0<q\le1$, rewrite it as
\[
 \left(1-q-\frac{q(q+16)}{6n}\right)^2
 +\frac{q^2(140+88q-q^2)}{36n^2}+\frac{15q^3}{n^3}>0.
\]
The claim and Proposition~\ref{prop:quartic-moment-gap} prove the result.
\end{proof}

\section{Explicit bounds in every complex dimension}\label{sec:alldim}
Here $N=2n$, $m=2n-1$, and $d=4m$ is a barrier parameter. This
section gives Hessian levels linear in $n$ under a quadratic spectral
threshold; Section~\ref{sec:freshuniform} trades these levels for a
uniform curvature condition.
\begin{theorem}[Spectral form]\label{an:thm:conditional}
Let $n\ge1$ and suppose
\begin{equation}\label{eq:potentialallconditions}
 \II\ge2g,\qquad \Hess V\ge2|\nabla V|g,\qquad \ell(V)\ge64m^2.
\end{equation}
Then $H>b(s)g$ for
\begin{equation}\label{an:eq:barrier}
 b(s)=2r-2d+1-2d\log\frac{r+d}{2d},\qquad r=\sqrt{1+s}.
\end{equation}
Moreover $\ell(V)\ge L$ implies $H>\eta g$ and
$\lambda_2-\lambda_1>2n+2+2\eta$ for each pair
\[
 (L,\eta)\in\{(64m^2,4m/7),(256m^2,5m),(6400m^2,50m)\},
\]
that is, $\lambda_2-\lambda_1$ exceeds $(30n+6)/7$, $22n-8$ and
$202n-98$, respectively.
\end{theorem}

\subsection{Retaining the diagonal weights in the trace estimate}
At a nonnegative null contact of $H-b(s)g$, use the notation of
Section~\ref{sec:commonlongitudinal}. In particular,
\[
 \mathcal B\ge0,\quad T=\tr\mathcal B=\lambda+s-Nb,
 \quad 0\le x\le\tau\le s.
\]
The sole first-derivative choice is $D_0=-2\alpha w$, and its exact cost
is \eqref{an:eq:contact}. Differentiating \eqref{an:eq:barrier} gives
\[
 \alpha=\frac1{r+d},\qquad
 \beta=-\frac1{2r(r+d)^2},\qquad
 A=\alpha+2s\beta=\frac{dr+1}{r(r+d)^2}>0.
\]
Write $\widehat b(r)=b(r^2-1)$. Then
$\widehat b'(r)=2r/(r+d)>0$ and $\widehat b(d)=1$.
The inequalities $1-1/y\le\log y\le y-1$ imply
\begin{equation}\label{an:eq:brackets}
 r-d+1\le\widehat b(r)
 \le2r-4d+1+\frac{4d^2}{r+d}<2r.
\end{equation}
For the last inequality, the difference from $2r$ is
$1-4dr/(r+d)<0$, since $r\ge1$ and $d\ge4$. Consequently
\[
 \alpha^2+b\beta=\frac{2r-b}{2r(r+d)^2}>0.
\]
The last two linear matrix terms in \eqref{an:eq:contact} are
nonpositive because the contact matrix $\mathcal B$ is positive
semidefinite.

\begin{remark}[Choice of the barrier]\label{rem:barrierconstruction}
The barrier \eqref{an:eq:barrier} is the case $\omega=d$ of the
construction in \cref{rem:barrierdesign}: it is the antiderivative of
$\alpha=1/(\sqrt{1+s}+d)$ normalized by $b(d^2-1)=1$. With a single
parameter, $d=4m$ satisfies the conditions $d\ge4$, $d\ge2N$, and
$N\le d/3$ for $m\ge3$ used in \cref{an:lem:sign}, and gives
$W\ge1/d^2$ on $1\le r\le d$; the threshold $4d^2=64m^2$ rounds up the
leading-order requirement $\lambda>\frac{1+\sqrt{29}}2d^2$ of that
range.
\end{remark}

\begin{lemma}[The weighted trace estimate]\label{an:lem:weightedtrace}
For every symmetric $m\times m$ matrix $\mathcal B$ and
$|w|^2\le s$,
\begin{equation}\label{an:eq:trace}
 \alpha\tr\mathcal B^2+2\beta|\mathcal Bw|^2
 \ge W(\tr\mathcal B)^2,
 \qquad
 W=\frac1{A^{-1}+(m-1)\alpha^{-1}}
 \ge\frac{d}{(r+d)(r+m d)}.
\end{equation}
One also has $W\ge A/m$. If $m\ge3$ and $r\ge d$,
then $W\ge6/(r+d)^2$.
\end{lemma}
\begin{proof}
In a basis adapted to $w$ (arbitrary if $w=0$), the diagonal
weights are $A_\tau=\alpha+2\beta|w|^2\ge A$ and $\alpha$;
the off-diagonal weights $\alpha+A_\tau$ and $2\alpha$ are positive.
The zero-linear-term case of Lemma~\ref{lem:traceminimum} (for $m\ge2$
and $w\ne0$), or in general weighted Cauchy--Schwarz after dropping
those squares, gives the first bound.
Since $A\le\alpha$,
the reciprocal denominator is at most $m/A$, proving
$W\ge A/m$. Since $A\ge d/(r+d)^2$, that denominator is
also at most $(r+d)(r/d+m)$, proving the last bound in
\eqref{an:eq:trace}. Finally, for $r\ge d=4m$,
\[
 \frac{d(r+d)}{r+m d}
 \ge\frac{2d}{m+1}
 =\frac{8m}{m+1}\ge6
 \quad(m\ge3).
\]
This calculation changes the basis only to estimate a quadratic form;
it imposes no alignment between the eigenvectors of $\mathcal B$,
the gradient, and the Hopf direction.
\end{proof}

By Lemma~\ref{an:lem:weightedtrace} and the nonpositivity of the two
linear terms noted above, the common cost therefore satisfies
\begin{equation}\label{an:eq:master}
 \mathcal C_D
 \le D_*+[1-4\alpha(1+2b\alpha)]\tau+(3-12\alpha)x-WT^2.
\end{equation}
The trace $T$ here is the trace of the actual contact matrix, rather
than an independently optimized scalar.

\begin{lemma}[The first-barrier sign in every direction]\label{an:lem:sign}
For $d=4m$, $N=m+1$, and $\lambda\ge4d^2$, the
right side of \eqref{an:eq:master} is strictly negative at every
admissible contact, for every $r\ge1$.
\end{lemma}
\begin{proof}
First suppose $1\le r\le d$. Monotonicity and
\eqref{an:eq:brackets} give $-d\le b\le1$ and $|b|\le d$.
Thus $T\ge\lambda-N>0$ and $T\le\lambda+d^2+Nd$. Moreover,
\[
 b^2+Nb\le d^2+N,\qquad
 -2\alpha bT\le2T,\qquad
 -2\beta b^2s\le r\le d.
\]
The two angular terms in \eqref{an:eq:master} are at most
$(4+8/d)s$. Also
$W\ge A/m\ge1/(4dm)=1/d^2$ on this range.
Summing these bounds with $2T\le2\lambda+2d^2+2Nd$ and $s\le d^2$, and
discarding the other nonpositive terms of $D_*$, yields
\[
 \mathcal C_D\le
 \lambda+7d^2+2Nd+N+9d-\frac{(\lambda-N)^2}{d^2}.
\]
This bound decreases with $\lambda\ge4d^2$, since its derivative is
$1-2(\lambda-N)/d^2<0$. At $\lambda=4d^2$ it is
\[
 -5d^2+2Nd+9N+9d-N^2/d^2
 \le-4d^2+\frac{27}{2}d
 \le-\frac58d^2<0,
\]
where $N\le d/2$ and $d\ge4$.

Now suppose $r\ge d$. Then $1\le b<2r$. The coefficient of $x$
in \eqref{an:eq:master} is positive, and the sum of the coefficients
of $\tau$ and $x$ is at least $4-12/d\ge1$. Their total contribution
is therefore at most $(4-16\alpha-8b\alpha^2)s$. Substitution of
$T=\lambda+s-Nb$ makes the scalar part the following convex
quadratic in $b$:
\begin{align}\label{an:eq:scalarU}
 U(b)={}&b^2\left(1+\frac{N}{r+d}
                     +\frac{s}{r(r+d)^2}\right)\notag\\
 &+b\left(N-\frac{2\lambda+2s}{r+d}
                 -\frac{8s}{(r+d)^2}\right)
       -\lambda+3s-\frac{(N+18)s}{r+d}.
\end{align}
The trace satisfies $T\ge\lambda+s-2Nr>0$, because $r\ge d\ge2N$.
For fixed $b\in[0,2r]$, both $U(b)$ and
$-W(\lambda+s-2Nr)^2$ decrease with $\lambda$. Hence
\[
 \mathcal C_D
 \le U(b)\big|_{\lambda=4d^2}
       -W(4d^2+s-2Nr)^2.
\]
The trace bound on the right is independent of $b$, so convexity
reduces the remaining scalar test to $b=0$ and $b=2r$.

If $m=1$, then $d=4$, $N=2$, and
$W\ge4/(r+4)^2$. Multiplying the two scalar upper bounds by
$(r+4)^2$ gives, respectively,
\begin{align*}
 P_0(r)&=-r^2(r-18)^2-343r^2+1500r-16868,\\
 P_2(r)&=-r^2(r-26)^2-115r^2+568r-16868.
\end{align*}
The elementary square inequalities
\[
 1500r\le100r^2+5625,
 \qquad 568r\le64r^2+\frac{5041}{4}
\]
show that the residual quadratics are bounded above by
$-243r^2-11243$ and $-51r^2-62431/4$, respectively. Thus both
upper bounds are strictly negative.

If $m\ge3$, then $d\ge12$, $N\le d/3$, and $N+1\le d/2$.
Discarding the negative terms $-8bs/(r+d)^2$ and
$-(N+18)s/(r+d)$ in \eqref{an:eq:scalarU}, and using
$s/[r(r+d)^2]\le1/(r+d)$, gives
\[
 \overline U(b)=b^2\left(1+\frac{N+1}{r+d}\right)
   +b\left(N-\frac{2(4d^2+s)}{r+d}\right)-4d^2+3s.
\]
Set $y=r/d\ge1$ and $q_*(y)=y^2-2y/3+63/16>0$. Then
\[
 \frac{4d^2+s-2Nr}{d^2}\ge q_*(y),
 \qquad
 \frac{W(4d^2+s-2Nr)^2}{d^2}
 \ge\frac{6q_*(y)^2}{(y+1)^2}.
\]
At $b=0$, the normalized scalar bound is
$3y^2-4-6q_*(y)^2/(y+1)^2$. Multiplication by $(y+1)^2$ gives
\[
 -3y^2(y-7/3)^2-\frac{415}{12}y^2
          +\frac{47}{2}y-\frac{12419}{128}.
\]
For $b=2r$, direct collection of terms gives
\[
 \overline U(2r)=3r^2+\frac{4(d+N+1)r^2}{r+d}+2Nr
 -4d^2\left(1+\frac{4r}{r+d}\right)+\frac{4r}{r+d}-3.
\]
Use $d+N+1\le3d/2$, $2N\le2d/3$, and
$4r/[d^2(r+d)]\le1/4$, and discard the negative term $-3/d^2$.
The normalized upper bound, after multiplication by $(y+1)^2$, is
\[
 -3y^2(y-31/9)^2-\frac{641}{27}y^2
          +\frac{26}{3}y-\frac{12387}{128}.
\]
Finally, completing one square in each linear term gives
\[
 \frac{47}{2}y\le25y^2+\frac{2209}{400}<25y^2+6,
 \qquad
 \frac{26}{3}y\le9y^2+\frac{169}{81}<9y^2+3.
\]
The two residual quadratics are therefore bounded above by
$-115y^2/12-11651/128$ and
$-398y^2/27-12003/128$. Both are strictly negative.
The cases $m=1$ and $m\ge3$ exhaust the odd dimensions
$m=2n-1$. This proves the sign for the entire gradient range
and every spectral value above the threshold.
\end{proof}

\subsection{The first continuation and the constant stages}
The barrier has the asymptotic expansion
\[
 b(s)=2\sqrt{s}-2d\log\sqrt{s}+O(1),\qquad s\longrightarrow\infty.
\]
Apply Proposition~\ref{prop:strongcontinuation} with $\kappa=2$.
Lemmas~\ref{lem:collar} and~\ref{lem:initial} apply with $a=2$
and logarithmic coefficient $2d>0$. The Laplacian eigenvalues of the members of the nested family satisfy
$\lambda_1\ge64m^2$, since $\lambda_1^0\ge\ell(V)$ by Lemma~\ref{lem:endpointspectral} and $\lambda_1^0$ is monotone under inclusion. The positive logarithmic term in the
boundary collar includes the equality case $\II=2g$.
Lemma~\ref{an:lem:sign} and Lemma~\ref{lem:exacttangent} exclude the
first interior null contact on each fixed member of this family, as in
the first stage of Lemma~\ref{lem:twostageaudit}.
Thus $H>b(s)g$ holds on the entire family before any constant-barrier
continuation is begun.

\begin{lemma}[A constant-stage criterion]\label{an:lem:constantcriterion}
Suppose the first barrier has already been established on a complete
continuation family from Proposition~\ref{prop:strongcontinuation},
whose initial member satisfies the constant bound as well by
Lemma~\ref{lem:initial}. Let $\eta>0$ and $R>1$ satisfy
\[
 b(R^2-1)\ge\eta,\qquad
 \lambda_1\ge\eta^2+N\eta+3(R^2-1)
\]
on that family. Then $H>\eta g$ on its final member.
\end{lemma}
\begin{proof}
At a nonnegative null contact of $H-\eta g$, the established strict
first barrier gives $b(s)<\eta$. Since $b$ is increasing, $r<R$.
No first-derivative choice is needed. The unrotated contact cost satisfies
\[
 \mathcal C_D
 =\eta^2+N\eta-\lambda-a^2-3\mathcal B(h,h)+3x
 \le\eta^2+N\eta-\lambda+3s
 <\eta^2+N\eta-\lambda+3(R^2-1)\le0.
\]
The collar and simultaneous initialization from
Lemmas~\ref{lem:collar} and~\ref{lem:initial}, followed by the second
continuation in Lemma~\ref{lem:twostageaudit}, exclude this contact.
\end{proof}

For $\lambda_1\ge64m^2$, choose
$\eta=4m/7$ and $R=32m/7-1$. The lower bound in
\eqref{an:eq:brackets} gives $b(R^2-1)\ge\eta$, while
\[
 \eta^2+N\eta-64m^2+3(R^2-1)
 =-\frac{20}{49}m^2-\frac{188}{7}m<0.
\]
For $\lambda_1\ge256m^2$, choose $\eta=5m$ and
$R=17m/2-1/2$. The bound
$e^{1/2}>1+1/2+1/8>25/16$ gives $\log(25/16)<1/2$, so
\[
 b((17m/2)^2-1)
 =(9-8\log(25/16))m+1>5m+1.
\]
Since $\widehat b'(r)<2$, decreasing $r$ by $1/2$ leaves
$b(R^2-1)>5m$. The constant-stage bound is
\[
 \eta^2+N\eta-256m^2+3(R^2-1)
 =-\frac{37}{4}m^2-\frac{41}{2}m-\frac94<0.
\]
Finally, for $\lambda_1\ge6400m^2$, take $\eta=50m$
and $R=71m/2-1/2$. The exponential series through degree four
at $13/8$ is greater than $79/16$; hence
\[
 b((71m/2)^2-1)
 =(63-8\log(79/16))m+1>50m+1.
\]
The same derivative bound gives $b(R^2-1)>50m$, and
\[
 \eta^2+N\eta-6400m^2+3(R^2-1)
 =-\frac{277}{4}m^2-\frac{113}{2}m-\frac94<0.
\]
Each argument uses its own fixed positive constant only after the
common first continuation has finished. Proposition~\ref{prop:commonbochner}
then gives the three gap bounds in Theorem~\ref{an:thm:conditional}.

Retaining the distinguished diagonal weight $A$ and the other
$m-1$ weights $\alpha$ yields $W\ge6/(r+d)^2$ in the
high-gradient range; the weaker $A/m$ alone does not give
this argument. Combined only with \eqref{an:eq:geomspectral}, this
construction gives sufficient spectral and boundary-curvature scales
$O(n^2)$ and $O(\sqrt n)$, respectively; Section~\ref{sec:freshuniform}
improves the first-stage spectral constant and gives a uniform level.

\begin{proof}[Proof of Theorem~\ref{an:thm:conditional}]
For $V=0$ the first continuation and Lemma~\ref{an:lem:constantcriterion}
give both stages on the nested family. For general $V$ use
Proposition~\ref{prop:fixedpotentialcontinuation} with $c_b=2$ and $L_0=L$, since
$b'=1/(\sqrt{1+s}+d)$: at a first-barrier contact the algebra is
\eqref{an:eq:contact} with $\lambda=\mu_t(x_0)\ge\ell(V)$, and
Lemma~\ref{an:lem:sign} applies; at a constant-stage contact the cost is
below $\eta^2+N\eta-\ell(V)+3(R^2-1)<0$ by the computation above, and
the potential correction is nonpositive.
Proposition~\ref{prop:commonbochner} gives the gaps.
\end{proof}

\begin{proof}[Proof of Theorem~\ref{thm:schall}]
The alternative statement is Theorem~\ref{an:thm:conditional} with $L=Am^2$.
Write $\kappa=\kappa^{\rm suff}_{n,A}$; both thresholds exceed $2$, since $\sqrt A\ge8$, $m\ge\sqrt n$ and $m\ge n$ give $\sqrt A\,m/(2\sqrt n)\ge4$ and $\sqrt{Am^2+4n}/n>\sqrt A\,m/n\ge8$. If $\kappa\ge\sqrt{Am^2+4n}/n$,
Lemma~\ref{lem:smooth-radial-bound} gives $\lambda_1^0>n^2\kappa^2\ge
Am^2+4n$; if $\kappa\ge\sqrt A\,m/(2\sqrt n)$, then
\eqref{an:eq:geomspectral} gives $\lambda_1^0\ge4n(1+\kappa^2)\ge Am^2+4n$.
In either case $\ell(V)\ge\lambda_1^0-\operatorname{osc}V\ge Am^2$, and
Theorem~\ref{an:thm:conditional} applies. The dimension-free conditions
follow from $(2\sqrt A)^2n^2-(Am^2+4n)=(4A-4)n-A>0$.
\end{proof}

\section{A uniform curvature condition from two rotations}\label{sec:freshuniform}
The extra curvature of the Hopf plane supplies a linear Hessian term
that is useful before scalar estimation. Retaining that term reduces
the contact problem to three gradient directions, without sign
assumptions on the angular coefficients.

Throughout this section, $m=2n-1$ denotes the transverse dimension and
$N=m+1=2n$ the real dimension. Thus the contact identities of
Section~\ref{sec:commonlongitudinal} are used with real dimension $N$.
For $n\ge2$, set
\begin{equation}\label{geo:parameters}
 d=\frac{9m}{10},\qquad \omega=\frac m2,\qquad
 L_m=\frac{11}{5}m^2+12m,\qquad r=\sqrt{1+s},\qquad \rho=r+\omega,
\end{equation}
and define
\begin{equation}\label{geo:barrier}
 b(s)=2r-2d+1-2\omega\log\frac{r+\omega}{d+\omega},\qquad
 \alpha=b'(s)=\frac1\rho,\qquad
 \beta=b''(s)=-\frac1{2r\rho^2}.
\end{equation}

\begin{theorem}[A common spectral threshold]\label{thm:freshspectral}
Under the standing domain hypotheses, for $n\ge2$,
\[
 \II\ge2g,\qquad \ell(V)\ge L_m,\qquad
 \Hess V\ge2|\nabla V|g
 \quad\Longrightarrow\quad
 \Hess(-\log\phi_1)>\max\{b(s),1\}g.
\]
\end{theorem}

\begin{remark}[Construction of the barrier]\label{rem:barrierdesign}
Write $\widehat b(r)=b(r^2-1)$, so that $\alpha=\widehat b'/(2r)$. Three
requirements determine the form of \eqref{geo:barrier}.

(i) \emph{Large gradients.} Under $\II\ge2g$, the collar
(\cref{lem:collar}) requires $b=a\sqrt s-c\log\sqrt s+O(1)$ with
$0<a\le2$ and $c>0$; we take the largest slope $a=2$. Then
$c_b=\sup2\sqrt s\,\alpha=2$ in
\cref{prop:fixedpotentialcontinuation}, which is the factor $2$ in
$\Hess V\ge2|\nabla V|g$. \cref{geo:three} needs $A=\alpha+2\beta s>0$,
and the low-gradient estimate uses $\varkappa=\alpha^2+b\beta>0$.

(ii) \emph{Dimension.} When $r$ is of order $m$, the terms $b^2$,
$Nb$, $s$, $\alpha Nb^2$ and $WT^2$ of the contact cost are all of order
$m^2$. We therefore take $b=m\,\zeta(r/m)+O(1)$ for a fixed profile $\zeta$; then $m^{-2}$ times
the cost has a limit at fixed $y=r/m$, and the barriers for different
$m$ differ only by the additive shift $1/m$ in $b/m$. This is the
limiting profile of Appendix~\ref{app:freshsigns}.

(iii) \emph{Simplest profile.} The slope $\widehat b'=2r/(r+\omega)$,
$\omega=wm$, rises from $0$ to $2$ on the scale $\omega$. It gives
$\alpha=1/\rho$, $A=(\omega r+1)/(r\rho^2)>0$,
$\varkappa=(2r-b)/(2r\rho^2)$, positive when $b<2r$, and the weight
$W=\omega/(\rho(r+m\omega))$ of \cref{geo:domination}; all are rational
in $r$, which keeps the sign verification elementary. Integrating and
normalizing by $b(d^2-1)=1$ gives \eqref{geo:barrier}.

The scale $\omega=m/2$ and the cutoff $d=9m/10$, where the constant
stage begins, are chosen so that the two stages have comparable
leading-order requirements. The control $\sigma=-2\alpha q/5$ below is
the transverse rotation that contributes at leading order when $r/m$
is fixed. The coefficient $11/5$ in $L_m$ gives a strict margin in both
stages, and the offset $12m$ absorbs the finite-dimensional corrections
in Appendix~\ref{app:freshsigns}. The barrier \eqref{an:eq:barrier} of
Section~\ref{sec:alldim} is the one-parameter case $\omega=d$.
\end{remark}

\subsection{The three gradient directions}
Lemma~\ref{geo:three} is used with $\nu=m$.

For $\eta=3$, the three minima are
\begin{align}\label{geo:endpoint-minima}
 P_v&=\frac{\alpha T^2}{m}+\frac{3T}{m}
                   -\frac{9(m-1)}{4\alpha m},\notag\\
 P_\perp&=
 \frac{(T+\ell/(2A)+3/(2\alpha))^2}{1/A+(m-1)/\alpha}
                   -\frac{\ell^2}{4A}-\frac9{4\alpha},\\
 P_H&=\frac{\alpha A T^2+\alpha(\ell+3)T
                   -(m-1)(\ell+3)^2/4}{\alpha+(m-1)A},
 \qquad A=\alpha+2\beta s,\quad\ell=4\varkappa s.\notag
\end{align}
They correspond respectively to a gradient parallel to the null
direction, perpendicular to its complex line, and in its Hopf direction.
They bound the objective from below for every nonnegative $B$, although
the minimizing matrices in \eqref{geo:endpoint-minima} need not be
nonnegative.

\subsection{The contact cost and its endpoint bounds}
For the moment let $V=0$, and let $\mu\ge L_m$ be the spectral value.
At a nonnegative null contact of $H-bg$, write
\[
 H=\Hess(-\log\phi_1),\qquad p=\nabla(-\log\phi_1)=av+w,
 \qquad |v|=1,\quad w\perp v,\quad h=Jv,
\]
\[
 B=(H-bg)|_{v^\perp}\ge0,\qquad T=\tr B=\mu+s-Nb,
 \qquad \tau=|w|^2,\quad q=\langle p,h\rangle,\quad x=q^2.
\]
Here $0\le x\le\tau\le s$. The barrier is increasing,
$b(d^2-1)=1$, and $-d<b<2r$. Indeed $\log u\ge1-u^{-1}$ gives
\[
 b\le2r-2d+1-\frac{2\omega(r-d)}{r+\omega},\qquad
 2r-b\ge\frac{2r(d+\omega)}{r+\omega}-1>0.
\]
For the lower bound, let $\widehat b(r)$ denote the radial expression
in \eqref{geo:barrier}, extended to $r\ge0$. It is increasing there.
Using $\log(1+t)\ge2t/(2+t)$ for $t\ge0$ (the difference vanishes at $0$ and has derivative $t^2/[(1+t)(2+t)^2]\ge0$) with $t=d/\omega$ gives
\[
 \widehat b(r)\ge\widehat b(0)\ge1-\frac{2d^2}{d+2\omega}>-d.
\]
It follows that
\[
 \varkappa:=\alpha^2+b\beta>0,\qquad
 \alpha+2\beta s=\frac{\omega r+1}{r\rho^2}>\frac\omega{\rho^2}>0,
\]
\[
 T>\mu-N^2-1+(r-N)^2>0,
 \qquad L_m-N^2-1=\frac65m^2+10m-2>0.
\]

Choose the longitudinal and transverse controls
$A_0=2\alpha w$ and $\sigma=-2\alpha q/5$.
In \eqref{an:eq:contact} the transverse control adds
\[
 x\left(\frac4{25}\alpha^2T-\frac45(m-1)\alpha\right)
              -\frac4{25}\alpha^2xB(h,h).
\]
Discard only its last, nonpositive term. The contact cost then satisfies
\begin{equation}\label{geo:contact-cost}
 \mathcal C\le D_*+C_\tau\tau+C_x\,x
 -\{\alpha\tr B^2+2\beta|Bw|^2+4\varkappa B(w,w)+3B(h,h)\},
\end{equation}
where
\begin{align*}
 D_*&=b^2+Nb-\mu-s
 -\alpha\{Nb^2+2bT+(N+2)s\}-2\beta b^2s,\\
 C_\tau&=1-4\alpha-8b\alpha^2,\qquad
 C_x=3-12\alpha+\frac4{25}\alpha^2T-\frac45(m-1)\alpha.
\end{align*}
In particular the Hopf term $-3B(h,h)$ is retained. When $w\ne0$,
take $e=w/|w|$ and $z=x/\tau$ in Lemma~\ref{geo:three}; the case
$w=0$ is its $\tau=0$ endpoint. Every contact cost is bounded by the
maximum of
\[
 D_*-P_v,\qquad D_*+C_\tau s-P_\perp,\qquad
 D_*+(C_\tau+C_x)s-P_H.
\]
This conclusion uses no sign assertion about $C_\tau,C_x$.

\begin{lemma}[Domination of the unrestricted endpoints]\label{geo:domination}
For the barrier \eqref{geo:barrier}, $r\ge d$ and $\mu\ge L_m$, put
\[
 W=\frac{\omega}{\rho(r+m\omega)},\qquad
 Z=\frac{s}{\rho}\left(1-\frac b{2r}\right),\qquad
 \overline Z=Z+\frac34\rho.
\]
Then
\[
 P_v>\frac\alpha mT^2,\qquad P_\perp>WT^2,\qquad
 P_H\ge WT^2+\frac{4\overline Z\{T-(m-1)\overline Z\}}{r+m\omega}.
\]
\end{lemma}
\begin{proof}
In this range $1\le b\le2r-2d+1\le2r-d$ and $0\le Z\le r$.
Completing a square in $r$ gives
\begin{align*}
 T-(m-1)Z
 &\ge L_m+Nd-1+r^2-(3m+1)r\\
 &\ge\frac{17m^2}{20}+\frac{57m}{5}-\frac54>0.
\end{align*}
Lower $A=\alpha+2\beta s$ to $A_{\rm low}=\omega/\rho^2$ in the
quadratic objective. This subtracts a nonnegative multiple of
$|Be|^2$ for every symmetric $B$, so it lowers each unrestricted
minimum. At the transverse endpoint, completing the diagonal squares
gives
\[
 P_{\perp,\rm low}-WT^2=
 \frac{4Z\{T-(m-1)Z\}+3\omega T+6\rho Z
                   -\frac94\rho\{r+(m-1)\omega\}}{r+m\omega}.
\]
The remaining sign follows from
\begin{align*}
 4\omega T-3\rho\{r+(m-1)\omega\}
 &\ge Q_m(r),\\
 Q_m(r)&=(2m-3)r^2-\left(\frac{11m^2}{2}+4m\right)r
             +\frac{109m^3}{20}+\frac{531m^2}{20}-2m>0.
\end{align*}
Indeed $b\le2r-2d+1$ gives $T\ge\Lambda+r^2-1-N(2r-2d+1)$, with
$\Lambda=L_m$, and with this lower bound the difference of the two sides
equals $m(m+1)(9m-10)/5\ge0$.
The leading coefficient is positive. Its discriminant is
\[
 -\frac m{20}(267m^3+2060m^2-7012m+480)<0,
\]
because the parenthesis is at least $1571m+480$ for $m\ge3$.
Thus $P_\perp\ge P_{\perp,\rm low}>WT^2$.
The same inequality implies $4\alpha T>3(m-1)$, proving the
assertion for $P_v$ from \eqref{geo:endpoint-minima}.
At the Hopf endpoint the entire linear term is retained. Substituting
$A_{\rm low}$ in \eqref{geo:endpoint-minima} gives
\[
 P_{H,\rm low}=WT^2+
 \frac{4\overline Z\{T-(m-1)\overline Z\}}{r+m\omega},
\]
and $P_H\ge P_{H,\rm low}$. No sign of the correction in this last
formula is required.
\end{proof}

The lemma is needed because nonnegative Hessian terms cannot be
discarded after replacing the constrained matrix problem by an
unrestricted minimum. It supplies direct comparisons for those
unrestricted minima. Hence, for $r\ge d$, the following three scalar
expressions bound all contact directions:
\begin{equation}\label{geo:scalar-costs}
\begin{aligned}
 E_0&=D_*-\frac\alpha mT^2,\\
 R_\perp&=D_*+C_\tau s-WT^2,\\
 R_H&=D_*+(C_\tau+C_x)s-WT^2
 -\frac{4\overline Z\{T-(m-1)\overline Z\}}{r+m\omega}.
\end{aligned}
\end{equation}

\subsection{The scalar signs and the two continuations}
\begin{proof}[Proof of Theorem~\ref{thm:freshspectral}]
For $1\le r\le d$, Lemma~\ref{lem:analytic-low} bounds the actual
nonnegative-matrix cost \eqref{geo:contact-cost} by a strictly negative
scalar envelope for every $\mu\ge L_m$.
For $r\ge d$, Lemmas~\ref{lem:analytic-zero} and~\ref{lem:analytic-high}
prove that the three expressions in \eqref{geo:scalar-costs} are
strictly negative at $\mu=L_m$. The high-range argument separates
the limit $1/m=0$ from its finite-dimensional correction: a scalar
calculus estimate proves the limiting sign, and differentiation in
$1/m$ gives a strict negative comparison for each of the Hopf and
transverse costs. Appendix~\ref{app:freshsigns} gives these comparisons.

The whole spectral halfline is covered with the barrier held fixed.
Since $b\ge1$, $T>0$, $\overline Z>0$, and $s\alpha^2<1$,
\begin{align*}
 (E_0)_\mu&=-1-2\alpha b-\frac{2\alpha T}{m}<0,\\
 (R_\perp)_\mu&=-1-2\alpha b-2WT<0,\\
 (R_H)_\mu&=-1-2\alpha b-2WT+
          \frac4{25}s\alpha^2-\frac{4\overline Z}{r+m\omega}<0.
\end{align*}
Lemma~\ref{geo:domination} also holds on that halfline.
Thus every first null contact has strictly negative cost and is
excluded by Lemma~\ref{lem:exacttangent}.

We now specify the continuation and its boundary control. For the
given potential, Lemma~\ref{lem:endpointspectral} gives
$\lambda_1^0(\Omega)\ge\ell(V)\ge L_m$. We first establish both
barriers for the Laplacian. Domain monotonicity keeps this lower bound
on the nested smooth
family of Proposition~\ref{prop:strongcontinuation}. The barrier satisfies
\[
 b(s)=2\sqrt s-m\log\sqrt s+O(1),\qquad
 \sup_{s\ge0}2\sqrt s\,b'(s)=2.
\]
Lemma~\ref{lem:collar} therefore supplies a uniform positive boundary
collar for $H-bg$ under $\II\ge2g$, including equality in the leading
boundary coefficient. The positive logarithmic coefficient $m$
is used in that equality case. It also supplies a collar for $H-g$.
Lemma~\ref{lem:initial} initializes both inequalities on a sufficiently
small domain. Run the first continuation over the entire family before
starting the constant-barrier stage, as in Lemma~\ref{lem:twostageaudit}.
The contact calculation proves $H>b(s)g$ throughout the first stage.

At a contact of $H-g$, this strict first barrier forces $s<d^2-1$.
Use the balanced controls of Proposition~\ref{prop:balancedcriterion}
with $c=3/10$ and $k=\sqrt{3+c^2}>7/4$.
Their identity $k^2-c^2=3$ cancels the Hopf coefficient.
Writing $S=d^2-1$, the two endpoint costs at $\mu=L_m$ obey
\[
 \mathcal C_0< -\frac{19}{10}m^2-10m+2<0,\qquad
 \mathcal C_H< -\frac{391}{10000}m^2-21m+13<0.
\]
For completeness, the first endpoint is
$-19291m^2/10000-1001m/100+91/50$.
For the second, $k>7/4$ and $\sqrt{d^2-1}>d-1$ give
$-391m^2/10000-2147m/100+611/50$.
Both decrease with $\mu$ and are negative for $m\ge3$.
The second continuation therefore proves $H>g$.

Finally run the fixed-domain $tV$ path in the same order, using
Proposition~\ref{prop:fixedpotentialcontinuation}.
Lemma~\ref{lem:endpointspectral} gives
$\mu_t(x)\ge\ell(V)\ge L_m$.
The additional first-stage contact cost is
\[
 -\frac t2\Hess V(v,v)+t\alpha\langle p,\nabla V\rangle\le0,
\]
since $\alpha\sqrt{s}\le1$ and $\Hess V\ge2|\nabla V|g$.
The constant stage adds only $-t\Hess V(v,v)/2\le0$.
The fixed-domain uniform collar and the two ordered continuations
complete the proof for the Schr\"odinger ground state.
\end{proof}

\begin{lemma}[The complex dimension one case]\label{lem:cp1constant}
Let $n=1$, $\II\ge2g$, $\Hess V\ge0$, and $\ell(V)\ge16$.
Then $\Hess(-\log\phi_1)>4g$ and $\gamma>12$.
\end{lemma}
\begin{proof}
At a nonnegative null contact of $H-\eta g$, the transverse space
is spanned by $h=Jv$. Thus $B(h,h)=\mu+s-2\eta$ and
$x=s-a^2$. The constant-barrier contact identity, including the
potential correction, reduces without any rotational test to
\[
 E_t=\eta^2+8\eta-4\mu_t-4a^2-\frac t2\Hess V(v,v).
\]
For $\eta=4$ and $\mu_t\ge16$ this is at most $-16$.
For $V=0$, Proposition~\ref{prop:strongcontinuation} (with $n=1$,
$\kappa=2$) and Lemma~\ref{lem:initial} give a small member of the
nested family with $H>4g$, and Lemma~\ref{lem:collar} gives a uniform
collar for the constant barrier. The first-loss argument in the proof
of Lemma~\ref{lem:twostageaudit}, applied to the single tensor $H-4g$,
together with Lemma~\ref{lem:exacttangent}, excludes a first null along
the family, whose spectral values are at least
$\lambda_1^0(\Omega)\ge\ell(V)\ge16$ by Lemma~\ref{lem:endpointspectral}.
The same single-tensor argument with $\Phi_t=\mathrm{Id}$ along the
fixed-domain path $tV$, where $\mu_t\ge\ell(V)$ by
Lemma~\ref{lem:endpointspectral} and $-\frac t2\Hess V(v,v)\le0$,
yields the result. No first
gradient-dependent barrier is needed in this dimension.
Proposition~\ref{prop:commonbochner} gives the gap.
\end{proof}

\begin{theorem}[A weaker dimension-independent curvature condition]
\label{thm:freshuniform}
Under the standing domain hypotheses in every complex dimension $n\ge1$,
assume
\[
 \II\ge\kappa_*g,\qquad \kappa_*:=\sqrt{44/5},\qquad
 \Hess V\ge2|\nabla V|g,\qquad \operatorname{osc}_\Omega V\le4n.
\]
Then $\Hess(-\log\phi_1)>g$; when $n=1$ one has the stronger bound
$\Hess(-\log\phi_1)>4g$. In particular
\[
 \lambda_2-\lambda_1>
 \max\left\{n+2+\sqrt{n^2+5n+11},\ 2n+\frac{14}{3}\right\}.
\]
\end{theorem}
\begin{proof}
For $n\ge2$, let $m=2n-1$ and $A=11/5$.
Corollary~\ref{cor:radialconversion} with $\kappa=\kappa_*\ge2$ gives
\[
 \ell(V)>R_m:=Am^2+(6A-1)m+9A+1+\frac1{4A}
 =L_m+\frac m5+\frac{104}{5}+\frac5{44}>L_m.
\]
Since $\kappa_*>2$, Theorem~\ref{thm:freshspectral} applies directly
in every dimension $n\ge2$.
For $n=1$, \eqref{an:eq:geomspectral} gives
$\lambda_1^0\ge4(1+\kappa_*^2)$, hence $\ell(V)\ge4\kappa_*^2=176/5>16$.
Lemma~\ref{lem:cp1constant} gives $H>4g$.
The gap follows from Corollaries~\ref{cor:oscillator-gap}
and~\ref{cor:quartic-linear-gap}, using
Proposition~\ref{prop:sharpdiameter}:
$D\le\arctan(2/\kappa_*)<\pi/2$.

\end{proof}

\begin{corollary}[Dimension-dependent sufficient levels]
\label{cor:improved-finitekappa}
For $n\ge2$, put $m=2n-1$ and $L_m=11m^2/5+12m$.
Under the potential hypotheses of Theorem~\ref{thm:freshuniform},
the boundary condition $\II\ge K_n(L_m)g$, with $K_n$ from
Corollary~\ref{cor:radialconversion}, suffices for $H>g$ and its stated
gap; $K_n(L_m)$ is the level $\kappa_n$ of \eqref{eq:intro-newkappa}.
\end{corollary}
\begin{proof}
Corollary~\ref{cor:radialconversion} supplies $\ell(V)>L_m$
independently of a Hessian estimate, and $K_n(L_m)\ge2$.
Apply Theorem~\ref{thm:freshspectral}, then Corollaries~\ref{cor:oscillator-gap}
and~\ref{cor:quartic-linear-gap} with Proposition~\ref{prop:sharpdiameter}
(note $K_n(L_m)\ge2>1$).
Finally, $\kappa_n<\kappa_*$. The function $g(\kappa)=(n+1)\kappa+\kappa^{-1}$
increases for $\kappa>(n+1)^{-1/2}$, and the second entry of
\eqref{eq:newradialthreshold} is the root of $g=\sqrt{L_m+4n}$ in this
range. By the proof of Theorem~\ref{thm:freshuniform},
$g(\kappa_*)^2-4n=R_m>L_m$, so that root is less than $\kappa_*$; also
$2<\kappa_*$.
\end{proof}

\begin{corollary}[A quartic diameter-sensitive uniform gap]
\label{cor:freshuniformgap}
Under Theorem~\ref{thm:freshuniform}, put
\[
 P_* =\frac{\pi^2}{\arctan^2(2/\kappa_*)},\quad
 a=n+2,\quad C=\frac{n^2+5n+11}{4},\quad
 E=\frac{n^2+10n+45}{24}.
\]
Let $y_0>0$ solve $2Ey_0^3+Cy_0^2=1$ and set
$y_*=\min\{y_0,P_*^{-1}\}$. Then
\[
 \lambda_2-\lambda_1>a+y_*^{-1}+Cy_*+Ey_*^2
 >a+2\sqrt C+\frac{(P_*-\sqrt C)_+^2}{P_*}.
\]
A larger independently established Hessian level $\eta$ may be used
by replacing $a$ with $n+1+\eta$,
$C$ with $(a^2+n+7)/4$, and $E$ with $[a(n+7)+n+31]/24$.
\end{corollary}
\begin{proof}
Use Proposition~\ref{prop:sharpdiameter} and Proposition~\ref{prop:quartic-moment-gap}.
The cubic has exactly one positive root, because its left side is
strictly increasing from zero to infinity.
\end{proof}

\begin{corollary}[Nonconvex potentials for $n\ge2$]
\label{cor:uniformnonconvex}
Let $n\ge2$, $m=2n-1$, $L_m=\frac{11}{5}m^2+12m$, $\II\ge2g$, and
$M\ge0$ (the case $M=0$ is Theorem~\ref{thm:freshspectral}). If
\begin{equation}\label{eq:uniformnonconvex}
 \ell(V)\ge L_m+M,\qquad
 \Hess V\ge\max\left\{2|\nabla V|-\tfrac{42}{25}M,\,
                       -\tfrac{91}{50}M\right\}g,
\end{equation}
then $\Hess(-\log\phi_1)>\max\{b(s),1\}g$ for the barrier
\eqref{geo:barrier}, and the gap bounds of
Theorem~\ref{thm:freshuniform} hold. In particular, if
$\II\ge\kappa g$ with $\kappa\ge2$ and
$\operatorname{osc}_\Omega V\le((n+1)\kappa+\kappa^{-1})^2-L_m-M$,
the spectral condition holds. For $\kappa=\kappa_*$ and
$\operatorname{osc}_\Omega V\le4n$ one may take $M=(m+104)/5$, which
permits $\Hess V\ge-\frac{42}{125}(m+104)\,g$ at critical points of
$V$; for
larger $\kappa$, $M$ grows like $(\kappa^2-\kappa_*^2)n^2$.
\end{corollary}
\begin{proof}
The case $M=0$ is Theorem~\ref{thm:freshspectral}. For $M>0$ we use
Proposition~\ref{prop:potentialmargin} with $c_b=2$ and $L_0=L_m+M$.
Its initialization $H_0>\max\{b,1\}g$ for the Laplacian on $\Omega$ is
Theorem~\ref{thm:freshspectral} with $V=0$, since
$\lambda_1^0\ge\ell(V)\ge L_m+M$ by Lemma~\ref{lem:endpointspectral}. In the proof of Theorem~\ref{thm:freshspectral} every
first-stage cost bound is negative at $\mu=L_m$ and has
$\mu$-derivative at most $-21/25$ on the whole halfline: for
$1\le r\le d$ this is the bound $Q_\mu<-21/25$ in the proof of
Lemma~\ref{lem:analytic-low}; for $r\ge d$ the displayed derivatives
of $E_0$, $R_\perp$ and $R_H$ are at most $-1$, $-1$ and
$-1+\frac4{25}s\alpha^2<-\frac{21}{25}$, since $s\alpha^2<1$. Hence at spectral
values $\mu\ge L_m+M$ the first-stage cost is below $-\frac{21}{25}M$.
The two balanced constant-stage endpoints are negative at $\mu=L_m$
and have $\mu$-coefficient $-(1-c^2)=-\frac{91}{100}$, so they are below
$-\frac{91}{100}M$. Thus $M_b=\frac{21}{25}M$ and
$M_\eta=\frac{91}{100}M$ in \eqref{eq:potentialmargins}, which is
\eqref{eq:uniformnonconvex}. The spectral condition follows from Drach's containment
$\Omega\subset B_{\operatorname{arccot}\kappa}(o)$ \cite[Theorem~A]{Drach},
Lemma~\ref{lem:rationalradial}, and
$\ell(V)\ge\lambda_1^0-\operatorname{osc}V$; for $\kappa=\kappa_*$ the
proof of Theorem~\ref{thm:freshuniform} gives
$((n+1)\kappa_*+\kappa_*^{-1})^2-4n-L_m=\frac m5+\frac{104}5+\frac5{44}$.
The gap transfer uses only $\Hess(-\log\phi_1)\ge g$, via
Corollaries~\ref{cor:oscillator-gap} and~\ref{cor:quartic-linear-gap}
and Proposition~\ref{prop:sharpdiameter} (here $\kappa\ge2>1$).
\end{proof}

\appendix
\section{Scalar estimates for the uniform contact inequality}
\label{app:freshsigns}
This appendix proves the scalar signs used in
Theorem~\ref{thm:freshspectral}. We use the notation of
Section~\ref{sec:freshuniform}. Thus $m=2n-1\ge3$, $N=m+1$, and
\[
 d=\frac{9m}{10},\qquad \omega=\frac m2,\qquad
 \Lambda=L_m=\frac{11m^2}{5}+12m,\qquad
 r=\sqrt{1+s},\qquad \rho=r+\omega.
\]
Write
\[
 \widehat b(r)=2r-2d+1-2\omega\log\frac{r+\omega}{d+\omega},
 \qquad b=\widehat b(r),\qquad
 \alpha=\rho^{-1},\qquad \beta=-\frac1{2r\rho^2},
 \qquad W=\frac{\omega}{\rho(r+m\omega)}.
\]
The quantities $D_*$, $C_\tau$, $C_x$, and $T=\mu+s-Nb$ are those
in \eqref{geo:contact-cost}; the scalar costs $E_0$, $R_\perp$, and
$R_H$ are defined in \eqref{geo:scalar-costs}.
When $b$ is varied algebraically, $r$, $\alpha$, and $\beta$ are held
fixed. The elementary bounds established in that section are
\[
 -d<\widehat b(r)<2r,
 \qquad \widehat b(d)=1,
 \qquad T>0\quad(\mu\ge\Lambda),
 \qquad 1\le\widehat b(r)\le2r-d\quad(r\ge d).
\]

\subsection{The low-gradient interval}
\begin{lemma}\label{lem:analytic-low}
For $1\le r\le d$ and $\mu\ge\Lambda$, the contact cost in
\eqref{geo:contact-cost} is strictly negative for every admissible
positive-semidefinite matrix $B$. More precisely, the upper envelope
\[
 Q(b,\mu)=D_*+4s+\frac{8ds}{\rho^2}
                +\frac{4sT}{25\rho^2}-WT^2
\]
is negative for every $b\in[-d,1]$.
\end{lemma}
\begin{proof}
For the actual barrier, the linear Hessian terms discarded from
\eqref{geo:contact-cost} are nonpositive because $B\ge0$ and
$\alpha^2+b\beta>0$. The quadratic Hessian terms satisfy
\[
 \alpha\operatorname{tr}B^2+2\beta|Bw|^2\ge WT^2.
\]
Indeed, their distinguished weight is at least
$\alpha+2\beta s\ge\omega/\rho^2$; weighted Cauchy--Schwarz on the
diagonal entries gives the reciprocal trace weight
$\{\rho^2/\omega+(m-1)\rho\}^{-1}=W$, and the off-diagonal terms
are nonnegative. In the angular terms, discard the negative multiples
of $x,\tau$ and use $-b\le d$, $T>0$, and $0\le x\le\tau\le s$.
This bounds the actual contact cost above by $Q$ before any
unrestricted matrix minimization.

Since $W\le1/(m\rho)$, the coefficient of $b^2$ in $Q$ is at least
$1-N/(m\rho)>0$. Thus $Q$ is convex in $b$.
On the present interval,
\[
 \frac m2<\rho\le\frac{7m}{5},\qquad
 W\ge\frac1{(m+9/5)\rho}\ge\frac5{8m\rho}.
\]
For $b\in[-d,1]$, $T\ge\mu-N$, and hence
\[
 Q_\mu\le-\frac{21}{25}+\frac{2(d-\rho WT)}{\rho}<-\frac{21}{25},
 \qquad
 \rho WT\ge\frac{5(\Lambda-N)}{8m}
 =\frac{11m}{8}+\frac{55}{8}-\frac5{8m}>d.
\]
It remains to set $\mu=\Lambda$ and to consider the two endpoints
of the algebraic $b$ interval.
At $b=-d$, put $a_-=2d/\rho+4s/(25\rho^2)$. Then
\[
 T=\frac{31m^2}{10}+\frac{129m}{10}+s,\qquad
 \rho(WT-a_-)\ge\frac{11m}{80}+\frac{129}{16}
                   +\frac{61s}{200m}>\frac\rho{12}.
\]
Regrouping $Q$, discarding $-(N+2)s/\rho$, and using
$s/(r\rho^2)\le1/\rho$, $s/\rho^2\le1$, and $s\le d^2$ gives
\begin{align*}
 Q(-d,\Lambda)
 &\le\frac{47}{12}d^2-\frac{13}{12}(\Lambda+Nd)
                  +\frac{d^2}{\rho}+8d\\
 &\le-\frac{223m^2}{1200}-\frac{1031m}{200}<0.
\end{align*}
At $b=1$, put $a_+=4s/(25\rho^2)$. Here
\[
 \rho(WT-a_+)\ge\frac{11m}{8}+\frac{55}{8}-\frac5{8m}
                   +\frac{61s}{200m}>\frac\rho2.
\]
Since $-N/\rho+s/(r\rho^2)\le0$, it follows that
\[
 Q(1,\Lambda)\le1+\frac32N-\frac32\Lambda+\frac52d^2+8d
       =-\frac{51m^2}{40}-\frac{93m}{10}+\frac52<0.
\]
Convexity in $b$ and monotonicity in $\mu$ prove the claim.
\end{proof}

\subsection{The zero angular endpoint}
\begin{lemma}\label{lem:analytic-zero}
For every $r\ge d$ and $\mu\ge\Lambda$, the actual barrier satisfies
$E_0(\widehat b(r),\mu)<0$.
\end{lemma}
\begin{proof}
First set $\mu=\Lambda$. Completing the trace terms gives
\[
 E_0=b^2\left(1-\frac1\rho+\frac{s}{r\rho^2}\right)
 +Nb-\Lambda-s-\frac{(\Lambda+s-b)^2}{m\rho}
 -\frac{(N+2)s}{\rho}.
\]
Since $s<r\rho$, this implies
\[
 E_0\le\mathscr A(b):=b^2+Nb-\Lambda-s
             -\frac{(\Lambda+s-b)^2}{m\rho}.
\]
The actual barrier obeys $1\le b\le b_*:=2r-d$.
The function $\mathscr A$ increases on $[0,b_*]$, because
$\Lambda+s-b\ge(r-1)^2+\Lambda+d-2>0$.
Set
\[
 C=\Lambda+d-1,\qquad D_0=\Lambda+Nd-d^2-1,
 \qquad a=\frac{309m^3}{100}+\frac{31m^2}{10}-\frac{263m}{5}+4.
\]
The following square completion is exact:
\begin{align*}
 -m\rho\mathscr A(b_*)={}&
 \frac12r^2(r-3m-4)^2+\frac12r^4
 +\left(\frac{59m}{5}-6\right)r^2+ar+C^2+\frac{m^2D_0}{2}.
\end{align*}
For $m\ge3$, $m^3\ge9m$ and $m^2\ge3m$ give $a>-16m$, while
$59m/5-6>9m$. The two potentially mixed terms therefore have sum
greater than $mr(9r-16)>0$, since $r\ge d\ge27/10$. Also $D_0>0$.
This proves the sign at $\mu=\Lambda$.
Finally,
\[
 \partial_\mu E_0=-1-2\alpha b-2\alpha T/m<0
\]
throughout the high-gradient range, which proves the stated half-line.
\end{proof}

\subsection{The limiting high-gradient cost}\label{geo:limiting-cost}
In this section $y\ge9/10$ and $t=y+1/2$. Define
\[
 z_0(y)=2y-\frac95-\log\frac{y+1/2}{7/5}.
\]
The normalized limiting Hopf cost will be
\[
 F_0(y,z)=z^2+a(t)z+b_0(t).
\]
Here the two scalar functions $a,b_0$ are displayed below. Since
$z_0\ge0$ and $a>0$, this cost increases with $z\ge0$.
The elementary global inequality
\begin{equation}\label{geo:log-lower}
 \log u\ge\frac{15(u-1)}{2(3u+5)},\qquad u\ge1,
\end{equation}
follows because the derivative of the difference is
$(3u-5)^2/[u(3u+5)^2]\ge0$ and the difference vanishes at $1$.

Put $t=y+1/2\ge7/5$ and define
\begin{align*}
 z(t)&=2t-\frac{53}{10}+\frac{28/3}{t+7/3},\\
 a(t)&=\frac{21}{25}+\frac4{25t}-\frac1{25t^2},\\
 b_0(t)&=-t^3+\frac{129}{25}t^2-\frac{551}{50}t
       +\frac{2921}{500}-\frac{13769}{2000t}+\frac{49}{500t^2},\\
 F(t)&=z(t)^2+a(t)z(t)+b_0(t).
\end{align*}
The function $z$ satisfies $z_0(y)\le z(y+1/2)$ and is the upper bracket for the normalized barrier
obtained from \eqref{geo:log-lower}. Direct grouping of the limiting contact
formula, with $\mathsf E_0=y^2+11/5$, gives the equivalent expression
\begin{equation}\label{geo:raw-limit}
 F=z^2+z\left(1-\frac{4y^2}{25t^2}\right)
     -\frac{\mathsf E_0^2}{t}+\frac{4y^2\mathsf E_0}{25t^2}
     +3y^2-\frac{11}{5}-\frac{9y^2}{5t}.
\end{equation}

\begin{lemma}\label{geo:limit}
For every $y\ge9/10$, one has
$F_0(y,z_0(y))\le F(y+1/2)<-1/10$.
\end{lemma}
\begin{proof}
For $t\ge2$, one has $z,z',z''\ge0$, $z'''<0$ and
$a>0$, $a'<0$, $a''>0$, $a'''<0$, as the factorizations
\[
 z=\frac{(5t-7)(12t+13)}{10(3t+7)},\quad
 z'=\frac{2(9t^2+42t+7)}{(3t+7)^2},\quad
 z'''=-\frac{4536}{(3t+7)^4},
\]
\[
 a'=-\frac{2(2t-1)}{25t^3},\quad
 a''=\frac{2(4t-3)}{25t^4},\quad
 a'''=-\frac{24(t-1)}{25t^5},\quad
 b_0'''=-6+\frac{41307}{1000t^4}-\frac{294}{125t^5}
\]
show. Moreover
\[
 z'\le2,\qquad z''=\frac{56/3}{(t+7/3)^3},\qquad
 a''\le\frac8{25t^3},\qquad b_0'''<-6+\frac{42}{t^4}.
\]
The product rule retains these signs without expanding a numerator:
\begin{align*}
 F'''&=6z'z''+(2z+a)z'''+3a'z''+3a''z'+a'''z+b_0'''\\
 &<\frac{224}{(t+7/3)^3}+\frac{48}{25t^3}-6+\frac{42}{t^4}\\
 &\le\frac{14}{5}+\frac6{25}-6+\frac{21}{8}
 =-\frac{67}{200}<0,
\end{align*}
using $224/(13/3)^3=6048/2197<14/5$ for $t\ge2$.
Thus $F'$ is strictly concave and $F''$ strictly decreases on $[2,\infty)$.

On $7/5\le t\le2$ use \eqref{geo:raw-limit}. Here
\[
 0\le z\le\frac65,\quad y^2\le\frac94,\quad
 \frac{y^2}{t}\ge\frac{81}{140},\quad
 \frac{\mathsf E_0}{t}\le\frac{89}{40},\quad
 \frac{\mathsf E_0^2}{t}\ge\frac{(301/100)^2}{7/5}.
\]
For the fourth bound, $\mathsf E_0/t=t-1+49/(20t)$ is convex, so its
maximum is at an endpoint. The functions $y^2/t$ and $\mathsf E_0^2/t$
increase on this interval; their derivative signs follow from
$t+1/2>0$ and $3t^2-t-49/20>0$, respectively. It follows that
\[
 F\le\frac{66}{25}+\frac{27}{4}-\frac{11}{5}
       -\frac{(301/100)^2}{7/5}
       -\left(\frac95-\frac4{25}\frac{89}{40}\right)\frac{81}{140}
 <-\frac1{10}.
\]

For the high range, $F''(3)<-9/5$, so $F''<-9/5$ on $[3,\infty)$.
At $t_*=19/5$ the bounds below give
\[
 F(t_*)<-\frac15,\qquad 0<F'(t_*)<\frac13.
\]
Taylor's inequality on either side of $t_*$ therefore yields
\[
 F(t)<-\frac15+\frac{|t-t_*|}{3}-\frac9{10}(t-t_*)^2
 \le-\frac15+\frac5{162}<-\frac1{10}.
\]
On $2\le t\le3$, the endpoint derivative bounds $F'(2)>3$ and
$F'(3)>3$, verified below, and concavity of $F'$ imply $F'>0$.
Thus $F(t)\le F(3)<-1/10$, completing the proof.
\smallskip
\noindent\emph{Scalar endpoint comparisons.}
For completeness, all comparisons used above can be made with the
following small rational brackets. At $t=2$,
\[
 \frac{17}{20}<z<1,\quad z'>\frac32,\quad a>\frac9{10},
 \quad a'=-\frac3{100},\quad b_0'>-\frac7{10}.
\]
Hence $F'=(2z+a)z'+a'z+b_0'>3$.
At $t=3$, one has $z=49/20$, $a=8/9$,
\[
 z'>\frac53,\quad a'z>-\frac1{25},\quad b_0'>-\frac{13}{2},
\]
which again give $F'>3$. To bound the second derivative at this point,
\[
 z'<\frac{67}{40},\quad 2z+a<6,\quad z''<\frac18,
 \quad 2a'z'<-\frac4{81},\quad a''z<\frac1{40},
 \quad b_0''<-\frac{49}{6}.
\]
Consequently
\[
 F''(3)<2\left(\frac{67}{40}\right)^2+\frac34-\frac4{81}
                  +\frac1{40}-\frac{49}{6}<-\frac95.
\]
At $t_*=19/5$ direct substitution into the four short rational
functions above gives
\begin{gather*}
 \frac{19}{5}<z<\frac{1911}{500},\qquad
 \frac{87}{100}<a<\frac{22}{25},\qquad b_0<-\frac{91}{5},\\
 \frac74<z'<\frac{351}{200},\qquad
 -\frac1{25}<a'z<0,\qquad
 -\frac{147}{10}<b_0'<-\frac{293}{20}.
\end{gather*}
In particular
\[
 F(t_*)<\left(\frac{1911}{500}\right)^2
       +\frac{22}{25}\frac{1911}{500}-\frac{91}{5}<-\frac15,
\]
and
\[
 0<\left(\frac{38}{5}+\frac{87}{100}\right)\frac74
       -\frac1{25}-\frac{147}{10}
 <F'(t_*)
 <\left(\frac{1911}{250}+\frac{22}{25}\right)\frac{351}{200}
       -\frac{293}{20}<\frac13.
\]
The point $19/5$ is chosen near the maximum of $F$, so its value and
slope can be controlled by the negative second derivative.

\end{proof}

\subsection{The finite-dimensional Hopf comparison}
\label{geo:finite-hopf}
The lower bound in Lemma~\ref{geo:domination} retains the term
$-3B(h,h)$ in the Hopf contact cost. We compare its normalized value
with $F_0$ from Lemma~\ref{geo:limit}. Throughout this subsection, put
\[
 y\ge\frac9{10},\qquad 0\le\varepsilon\le\frac13,\qquad
 \chi=y+\frac12,\qquad q=2y+1.
\]
For an algebraic variable $z$, define
\begin{gather*}
 \zeta=z+\varepsilon,\qquad
 S=y^2-\varepsilon^2,\qquad
 \mathsf E=\frac{11}{5}+12\varepsilon+S,\qquad
 \mathsf T=\mathsf E-(1+\varepsilon)\zeta,\\
 Z=\frac{S}{\chi}\left(1-\frac{\zeta}{2y}\right),\qquad
 \widehat Z=Z+\frac34\chi,\qquad K=1+2\varepsilon y.
\end{gather*}
The continuously varying parameters below describe algebraic costs;
only $\varepsilon=1/m$ and the actual barrier need represent a
geometric contact. Under the substitution
$\varepsilon=1/m$, $y=r/m$, and $b=m\zeta$, one has
$s=m^2S$. Thus, at $\mu=\Lambda$, the normalized Hopf cost equals
\[
 \frac K{m^2}\left\{D_*+(C_\tau+C_x)s-WT_{\rm un}^2
 -\frac{4\overline Z[T_{\rm un}-(m-1)\overline Z]}{r+m\omega}\right\}.
\]
Here $T_{\rm un}=m^2\mathsf T$ and $\overline Z=m\widehat Z$.
Its explicit rational extension to $\varepsilon=0$ is
\begin{align}
 \mathcal H(y,\varepsilon,z)={}&K\Biggl\{
 \zeta^2\left(1+\frac{1+\varepsilon}{\chi}
                   +\frac{\varepsilon S}{y\chi^2}\right)
 +\zeta\left(1+\varepsilon-\frac{2\mathsf E}{\chi}
                         -\frac{8\varepsilon S}{\chi^2}\right)\notag\\
 &\qquad-\frac{11}{5}-12\varepsilon+3S
       -\frac{(1+19\varepsilon)S}{\chi}
       +S\left(\frac{4\mathsf T}{25\chi^2}
                  -\frac{4(1-\varepsilon)}{5\chi}\right)\Biggr\}\notag\\
 &-\frac{\mathsf T^2}{\chi}
       -8\varepsilon\widehat Z
                      \bigl(\mathsf T-(1-\varepsilon)\widehat Z\bigr).
 \label{geo:hopf-normalized}
\end{align}
At $\varepsilon=0$, formula~\eqref{geo:hopf-normalized} is exactly
$F_0(y,z)$ defined above. We also write
\[
 z_0(y)=2y-\frac95-\log\frac{y+1/2}{7/5},\qquad
 z_W(y)=2y-\frac{14}{5}+\frac{14}{5q}.
\]
The elementary bound $\log u\ge1-u^{-1}$ for $u\ge1$ gives
\[
 0\le z_0(y)\le z_W(y)<2y.
\]
Indeed $z_0(9/10)=z_W(9/10)=0$, and both functions are increasing.

\begin{lemma}[The Hopf dimension correction]
\label{geo:hopf-comparison}
For $y\ge9/10$, $0\le\varepsilon\le1/3$, and
$0\le z\le z_W(y)$,
\begin{equation}
 \mathcal H(y,\varepsilon,z)
 \le(1-\varepsilon)F_0(y,z)-\frac{13}{2}\varepsilon.
 \label{geo:hopf-dimension-bound}
\end{equation}
Consequently, the strict scalar inequality $F_0(y,z_0(y))<0$ implies
$\mathcal H(y,\varepsilon,z_0(y))<0$ throughout this parameter range.
\end{lemma}

\begin{proof}
For $\varepsilon>0$, define
\[
 \Delta(y,\varepsilon,z)
 =\frac{\mathcal H(y,\varepsilon,z)
                    -(1-\varepsilon)F_0(y,z)}{\varepsilon}.
\]
The numerator vanishes at $\varepsilon=0$. Its denominators are
independent of $\varepsilon$, so $\Delta$ extends polynomially in
$\varepsilon$ to zero. Explicitly,
$1000y^2q^2\Delta=\sum_{i=0}^7\varepsilon^i(c_{i0}+c_{i1}z+c_{i2}z^2)$, where
{\small
\begin{align*}
 c_{00}&=-y^2(34720y^5-66240y^4+254784y^3+234152y^2+318360y+133955),\\
 c_{01}&=40y^2(368y^3+568y^2+1190y+395),\qquad
 c_{02}=1000y^2q(4y^2+4y-1),\\
 c_{10}&=-y^2(243600y^4+886720y^3+752920y^2+772600y+303333),\\
 c_{11}&=40y^2(1768y^3+2500y^2+2950y+1191),\qquad
 c_{12}=-2000y^2q^2,\\
 c_{20}&=16y^2(7760y^3+2400y^2+20474y+9535),\quad
 c_{21}=-80yq(200y^2-258y+145),\quad
 c_{22}=4000y,\\
 c_{30}&=80y(2920y^3+6484y^2+2526y-145),\quad
 c_{31}=-80y(684y^2+2900y+1175),\quad
 c_{32}=8000y(3y+1),\\
 c_{40}&=-80y(1068y^2+2700y+1225),\quad
 c_{41}=8000y(8y-1),\quad c_{42}=8000,\\
 c_{50}&=8000y(y-2),\quad c_{51}=16000q,\quad c_{52}=-8000,\quad
 c_{60}=8000(4y+1),\quad c_{61}=-16000,\\
 c_{62}&=c_{71}=c_{72}=0,\quad c_{70}=-8000.
\end{align*}}
All later statements about $\Delta$ follow from this polynomial by
differentiation and substitution. We first prove
\begin{equation}
 \partial_\varepsilon\Delta<0
 \quad\text{for }y\ge9/10,\quad
 0\le\varepsilon\le1/3,\quad 0\le z\le2y.
 \label{geo:hopf-dimension-monotonicity}
\end{equation}
Write $Q=\partial_\varepsilon\Delta$. Differentiating
\eqref{geo:hopf-normalized} gives
\[
 \partial_z^2Q=-\frac{4A}{y^2q^2},
\]
where
\[
 A=20\varepsilon^4-16\varepsilon^3
   -36\varepsilon^2y^2-12\varepsilon^2y-4\varepsilon y
   +4y^4+4y^3+y^2.
\]
The parameter bounds imply
\begin{align*}
 A&\ge4y^4+4y^3-3y^2-\frac83y-\frac{16}{27}\\
  &\ge\frac{96}{25}y^2-\frac83y-\frac{16}{27}
   \ge\frac{296}{375}y-\frac{16}{27}
   >\frac1{10}>0.
\end{align*}
In the second line we used $4y^2+4y-3\ge96/25$ and $y\ge9/10$.
Thus $Q$ is concave in $z$.

To identify its maximum, differentiate at $z=2y$. This gives exactly
\begin{align*}
 \frac{y^2q^2}{4}\partial_zQ(y,\varepsilon,2y)
 ={}&\frac{242}{25}y^5+17y^4+\frac{55}{2}y^3+\frac{1191}{100}y^2
 +\varepsilon\Bigl(-16y^4+\frac{316}{25}y^3+\frac{168}{25}y^2-\frac{29}{5}y\Bigr)\\
 &+\varepsilon^2\Bigl(\frac{774}{25}y^3-150y^2-\frac{141}{2}y\Bigr)
 +\varepsilon^3(64y^2+24y)+20\varepsilon^4-24\varepsilon^5.
\end{align*}
Discarding the nonnegative terms
$\varepsilon(\frac{316}{25}y^3+\frac{168}{25}y^2)$,
$\frac{774}{25}\varepsilon^2y^3$, $\varepsilon^3(64y^2+24y)$ and
$20\varepsilon^4$, and then using $\varepsilon\le1/3$ and $y\ge9/10$, gives
\begin{align*}
 \frac{y^2q^2}{4}\partial_zQ(y,\varepsilon,2y)
 &\ge\frac{242}{25}y^5+17y^4+\frac{55}{2}y^3
       +\frac{1191}{100}y^2
       -16\varepsilon y^4-150\varepsilon^2y^2\\
 &\qquad-\frac{141}{2}\varepsilon^2y
          -\frac{29}{5}\varepsilon y-24\varepsilon^5\\
 &\ge\frac{242}{25}y^5+\frac{35}{3}y^4+\frac{55}{2}y^3
       -\frac{1427}{300}y^2-\frac{293}{30}y-\frac8{81}\\
 &\ge9y^5+11y^4+27y^3-5y^2-10y-\frac1{10}>0.
\end{align*}
The last sign follows from
\[
 27y^3-5y^2-10y-\frac1{10}
 \ge\frac{193}{10}y^2-10y-\frac1{10}
 \ge\frac{737}{100}y-\frac1{10}>0.
\]
Concavity and $\partial_zQ(2y)>0$ show that $Q$ increases on $[0,2y]$.

For its value at the right endpoint, differentiation gives exactly
{\small
\begin{align*}
 q^2Q(y,\varepsilon,2y)={}&
 -\frac{3354}{25}y^4-\frac{17968}{25}y^3-\frac{13123}{25}y^2
 -\frac{16933}{25}y-\frac{303333}{1000}\\
 &+\varepsilon\Bigl(\frac{3008}{25}y^3+\frac{4448}{25}y^2
   +\frac{84616}{125}y+\frac{6468}{25}\Bigr)
 +\varepsilon^2\Bigl(\frac{16512}{25}y^2+\frac{6504}{25}y
   +\frac{1056}{25}-\frac{174}{5y}\Bigr)\\
 &+\varepsilon^3\Bigl(\frac{4256}{25}y-800-\frac{392}{y}\Bigr)
 +\varepsilon^4\Bigl(200+\frac{80}{y}\Bigr)
 +\frac{48\varepsilon^5}{y^2}-\frac{56\varepsilon^6}{y^2}.
\end{align*}}
Rounding each positive coefficient up and each negative coefficient
toward zero, and dropping the negative terms
$-\frac{174}{5y}\varepsilon^2$, $-\varepsilon^3(800+\frac{392}{y})$ and
$-\frac{56\varepsilon^6}{y^2}$, gives the small-integer upper bound
\begin{align}
 q^2 Q(y,\varepsilon,2y)\le{}&
 -134y^4-718y^3-524y^2-677y-303\notag\\
 &+\varepsilon(121y^3+178y^2+677y+259)\notag\\
 &+\varepsilon^2(661y^2+261y+43)+171\varepsilon^3y
   +\varepsilon^4\left(200+\frac{80}{y}\right)
   +\frac{48\varepsilon^5}{y^2}.
 \label{geo:hopf-integer-bound}
\end{align}
The five positive terms in
\eqref{geo:hopf-integer-bound} are bounded, in order, by
\[
 41y^3+60y^2+226y+87,\qquad
 74y^2+29y+5,\qquad7y,\qquad4,\qquad1.
\]
These follow from $\varepsilon\le1/3$ and $y\ge9/10$; for the last
two, use $40(5+20/9)/81<4$ and $48\cdot100/(243\cdot81)<1$.
Consequently
\[
 q^2Q(y,\varepsilon,2y)
 \le-134y^4-677y^3-390y^2-415y-206<0.
\]
This proves \eqref{geo:hopf-dimension-monotonicity} and hence
$\Delta(y,\varepsilon,z)\le\Delta(y,0,z)$.

It remains to bound the latter scalar quadratic. Its quadratic and
linear coefficients in $z$ are, respectively,
\[
 2y+1-\frac2{q}>0,
 \qquad
 \frac{92y}{25}+2+\frac{449}{25q}
                       -\frac{104}{25q^2}>0.
\]
Thus it increases for $z\ge0$. Substitution of the weak logarithmic
upper bound $z_W$, followed by rational simplification, gives
\begin{align}
 \Delta(y,0,z_W)\le{}&
 -\frac{17}{25}y^3+\frac{71}{5}y^2-\frac{8139}{100}y
       +45-\frac{260}{q}+\frac{113}{q^2}.
 \label{geo:hopf-weak-log}
\end{align}
The right side exceeds the exact expression by
\[
 \frac{21}{100}+\frac{3833}{1000q}
       +\frac{73}{125q^2}+\frac{3416}{125q^3}>0,
\]
so this simplification also leaves no sign to be checked.
The square completion
\[
 \frac{17}{25}y^2-\frac{71}{5}y+\frac{8139}{100}
 =\frac{17}{25}\left(y-\frac{355}{34}\right)^2
       +\frac{6169}{850}>7
\]
and $q\ge14/5$ therefore imply
\[
 \Delta(y,0,z_W)<-7y+45-\frac{260}{q}
                         +\frac{113}{q^2}
                   <-7y+45-\frac{219}{q}.
\]
Finally, the arithmetic--geometric mean inequality yields
\[
 7y+\frac{219}{q}
 =\frac72q+\frac{219}{q}-\frac72
 \ge\sqrt{3066}-\frac72>\frac{103}{2}.
\]
Therefore $\Delta(y,0,z_W)<-13/2$. Monotonicity in $z$ and
\eqref{geo:hopf-dimension-monotonicity} prove
\eqref{geo:hopf-dimension-bound}. At $\varepsilon=0$ that comparison
is equality; strict negativity of the endpoint follows from the
separate strict inequality for $F_0(y,z_0(y))$.
\end{proof}

\begin{remark}[What the comparison uses]
The factor $1-\varepsilon$ and the margin $13/2$ are chosen for the
above elementary bounds. Neither is optimized. Concavity of the differentiated quadratic gives the comparison
on the entire inverse-dimension interval, including its limiting endpoint.
\end{remark}
\subsection{The transverse dimension comparison}\label{geo:transverse-cost}
We use the normalized variables of the preceding subsection, with
$y\ge9/10$, $0\le\varepsilon\le1/3$, and $0\le z\le2y$.
At the actual barrier, Lemma~\ref{geo:domination} bounds the unrestricted
transverse minimum below by $WT_{\rm un}^2$.
Extend the resulting scalar upper bound to algebraic $z$ by setting
\[
 \mathcal U_\perp(y,\varepsilon,z)
 =\frac{K}{m^2}\{D_*+(1-4\alpha-8b\alpha^2)s-WT_{\rm un}^2\},
\]
where the unnormalized quantities are evaluated at
$m=1/\varepsilon$, $r=y/\varepsilon$, $b=(z+\varepsilon)/\varepsilon$,
and $T_{\rm un}=\mathsf T/\varepsilon^2$. Equivalently,
\begin{equation}\label{geo:transverse-normalized}
\begin{split}
 \mathcal U_\perp=K\bigg[&\zeta^2\left(1+\frac{1+\varepsilon}{\chi}
                  +\frac{\varepsilon S}{y\chi^2}\right)
 +\zeta\left(1+\varepsilon-\frac{2\mathsf E}{\chi}
                         -\frac{8\varepsilon S}{\chi^2}\right)\\
 &-\frac{11}{5}-12\varepsilon
        -\frac{(1+7\varepsilon)S}{\chi}\bigg]
        -\frac{\mathsf T^2}{\chi}.
\end{split}
\end{equation}
This formula also defines the value $\mathcal U_{\perp,0}(y,z)$ at
$\varepsilon=0$, with no singular substitution.
Put
\[
 D_\perp(y,\varepsilon,z)
 =\frac{\mathcal U_\perp(y,\varepsilon,z)
          -(1-\varepsilon)\mathcal U_{\perp,0}(y,z)}{\varepsilon},
\]
using its continuous extension at zero.
For the actual profile $z_0(y)$, the comparison is as follows.
\begin{lemma}[The transverse dimension correction]\label{geo:transverse-comparison}
For $y\ge9/10$ and $0\le\varepsilon\le1/3$,
\begin{equation}\label{geo:transverse-dimension-bound}
 \mathcal U_\perp(y,\varepsilon,z_0)
 \le(1-\varepsilon)\mathcal U_{\perp,0}(y,z_0)-24\varepsilon.
\end{equation}
\end{lemma}
\begin{proof}
\emph{The derivative in dimension.}
Write $Q=\partial_\varepsilon D_\perp$. Differentiating the raw
rational contact expression gives
\[
 Q_{zz}=\frac{4(8y^3-y-12\varepsilon^2y-4\varepsilon)}{yq^2}>0.
\]
Indeed its parenthesis is at least
$8y^3-7y/3-4/3\ge1799/750>0$.
It is enough to bound $Q$ at $z=0,2y$.
Define $P_0=-5yq^2Q(y,\varepsilon,0)/2$ and
$P_2=-5yq^2Q(y,\varepsilon,2y)/2$.
Direct differentiation of \eqref{geo:transverse-normalized} gives exactly
{\small
\begin{align*}
 P_0={}&180y^5+270y^4+348y^3+1421y^2+666y
 +\varepsilon(240y^4+800y^3-680y^2-540y)\\
 &-\varepsilon^2(660y^3+180y^2+180y)+\varepsilon^3(40-640y^2)+100\varepsilon^4y,\\
 P_2={}&220y^5+1030y^4+298y^3+1201y^2+666y
 +\varepsilon(-240y^4+800y^3-840y^2-540y)\\
 &-\varepsilon^2(1380y^3+180y^2+60y)+\varepsilon^3(40-320y^2)+100\varepsilon^4y.
\end{align*}}
Dropping the nonnegative terms and using $\varepsilon\le1/3$ gives the grouped estimates
\begin{align*}
 P_0\ge{}&180y^5+270y^4+348y^3+1421y^2+666y\\
 &-\frac{680y^2+540y}{3}
  -\frac{660y^3+180y^2+180y}{9}-\frac{640y^2}{27}\\
 >{}&180y^5+270y^4+274y^3+1150y^2+466y>0,\\
 P_2\ge{}&220y^5+1030y^4+298y^3+1201y^2+666y\\
 &-\frac{240y^4+840y^2+540y}{3}
  -\frac{1380y^3+180y^2+60y}{9}-\frac{320y^2}{27}\\
 >{}&220y^5+950y^4+144y^3+889y^2+479y>0.
\end{align*}
The three subtracted groups bound the negative $\varepsilon$,
$\varepsilon^2$, and $\varepsilon^3$ contributions, and the final bounds
round down to integers.
Thus $Q<0$ at both angular endpoints. Its convexity in $z$ gives
$Q<0$ on the entire interval, whence
$D_\perp(y,\varepsilon,z)\le D_\perp(y,0,z)$.

\smallskip
\noindent\emph{The weak logarithm bound.}
Let $D_0(y,z)=D_\perp(y,0,z)$. Direct differentiation yields
\[
 (D_0)_{zz}=\frac{2(8y^3+20y^2+10y-1)}{q^2}>0.
\]
The elementary inequality $\log u\ge1-u^{-1}$ gives
\[
 0\le z_0(y)\le z_W(y):=2y-\frac{14}{5}+\frac{14}{5q}<2y.
\]
By convexity, it suffices to bound $D_0$ at $0,z_W$.
At zero, partial fractions give
\[
 D_0(y,0)=-y^3-\frac32y^2-\frac{801}{20}y
             +\frac{37}{8}-\frac{26621}{200q}
 <-\frac{801}{20}\frac9{10}+\frac{37}{8}<-24.
\]
At the other endpoint the exact expression is
\begin{align*}
 D_0(y,z_W)={}&-y^3+\frac{113}{10}y^2-\frac{10277}{100}y
 +\frac{24621}{200}-\frac{3109}{8q}\notag\\
 &+\frac{3706}{25q^2}+\frac{112}{25q^3}-\frac{392}{25q^4}.
\end{align*}
Its cubic part has the square completion
\[
 -y^3+\frac{113}{10}y^2-\frac{10277}{100}y
 =-y\left\{\left(y-\frac{113}{20}\right)^2
                   +\frac{28339}{400}\right\}<-70y.
\]
Since $q\ge14/5$, the three positive terms after this cubic are
bounded above by $124$, $19$, and $1$, respectively.
Also $3109/8>384$. Consequently
\[
 D_0(y,z_W)<-70y+144-\frac{384}{q}
           =179-35q-\frac{384}{q}<-31,
\]
where the last inequality follows from
\[
 35q+\frac{384}{q}
 =\frac{35(q-3)^2}{q}+210+\frac{69}{q}>210.
\]
Thus $D_0(y,z)<-24$ for $0\le z\le z_W$.
Together with the derivative comparison, this proves
\eqref{geo:transverse-dimension-bound}.

\end{proof}

\paragraph{The limiting transverse cost.}
Let $F_0$ be the limiting Hopf cost of Lemma~\ref{geo:limit}. At
$\varepsilon=0$ the trace corrections from the retained linear
Hessian terms vanish, and the difference of angular coefficients gives
\[
 \mathcal U_{\perp,0}(y,z)=F_0(y,z)
 -y^2\left\{3+\frac{4(11/5+y^2-z)}{25\chi^2}-\frac4{5\chi}\right\}.
\]
For $0\le z\le2y$, the expression in braces is greater than
$3-4/7>0$, since $11/5+y^2-z\ge(y-1)^2+6/5$.
The established limiting sign $F_0(y,z_0)<0$ therefore gives
$\mathcal U_{\perp,0}(y,z_0)<0$. Inequality~\eqref{geo:transverse-dimension-bound}
proves the transverse sign for every dimension.

\begin{lemma}\label{lem:analytic-high}
For every $r\ge d$, the actual barrier satisfies
\[
 R_H(\widehat b(r),\Lambda)<0,
 \qquad R_\perp(\widehat b(r),\Lambda)<0.
\]
\end{lemma}
\begin{proof}
Set $\varepsilon=1/m$ and $y=r/m$. Then
$\widehat b(r)/m=z_0(y)+\varepsilon$ and $K=1+2\varepsilon y>0$.
The normalizations \eqref{geo:hopf-normalized} and
\eqref{geo:transverse-normalized} give
\[
 \frac K{m^2}R_H=\mathcal H(y,\varepsilon,z_0(y)),
 \qquad
 \frac K{m^2}R_\perp=\mathcal U_\perp(y,\varepsilon,z_0(y)).
\]
Lemma~\ref{geo:limit}, the limiting transverse sign
$\mathcal U_{\perp,0}(y,z_0)<0$ proved after
Lemma~\ref{geo:transverse-comparison}, and the comparisons in
Lemmas~\ref{geo:hopf-comparison} and~\ref{geo:transverse-comparison}
make both right sides strictly negative.
\end{proof}

\printbibliography

\end{document}